\documentclass{amsart}
\usepackage[utf8]{inputenc}
\usepackage[margin=1in]{geometry}  % set the margins to 1in on all sides
\usepackage{verbatim}
\usepackage{graphicx,wrapfig}              % to include figures
\usepackage{amsmath,amsfonts,amsthm,amssymb}
\usepackage{pifont}% http://ctan.org/pkg/pifont
\usepackage{tikz} 
 \usepackage{relsize}%use \mathlarger{} to make math symbols larger
\usetikzlibrary{shapes}
\usetikzlibrary{positioning}
\usepackage{ytableau,enumitem}   
\usepackage{xcolor}
\usepackage{float}
\usepackage{pifont}
\usepackage{todonotes}
\usepackage{multirow}
\newcommand{\NN}{\mathbb{N}}
\newtheorem{thm}{Theorem}[section]
\newtheorem{lemma}[thm]{Lemma}
\newtheorem{method}{Method}
\newtheorem{prop}[thm]{Proposition}
\newtheorem{corollary}[thm]{Corollary}
\newtheorem{claim}{Claim}
\definecolor{ultragreen}{RGB}{24,120,50} \theoremstyle{remark}
\newtheorem{example}[thm]{\color{ultragreen}Example}
\theoremstyle{plain}
\newtheorem{conj}[thm]{Conjecture}
\newtheorem{defn}[thm]{Definition}

\newenvironment{claimproof}[1]{\par\noindent\underline{Proof of Claim:}\space#1}{\leavevmode\unskip\penalty9999 \hbox{}\nobreak\hfill\quad\hbox{$\Diamond$}}
\numberwithin{thm}{section}

\newcommand{\fl}[1]{\lfloor #1 \rfloor}
\newcommand{\ce}[1]{\lceil #1 \rceil}
\newcommand{\MM}{\overline{\mathcal{M}}}
\newcommand{\OO}{\mathcal{O}}
\usetikzlibrary{patterns,decorations.pathreplacing}
\newcommand{\rrec}[2]{\filldraw[red]  (#2-.8,-{#1}+.8) rectangle (#2-.2,-{#1}-.8);
\draw[pattern=north east lines] (#2-.8,-{#1}+.8) rectangle (#2-.2,-{#1}-.8);
\draw (#2-.8,-{#1}+.8) rectangle (#2-.2,-{#1}-.8);
} %red rectangle starting at row arg. 1, col arg.2
\newcommand{\rrrec}[2]{\filldraw[red]  (#2-.8,0) rectangle (#2-.2,-{1}+.2);
\draw[pattern=north east lines]  (#2-.8,0) rectangle (#2-.2,-{1}+.2);
\draw  (#2-.8,0) rectangle (#2-.2,-{1}+.2);
\filldraw[red]  (#2-.8,-{#1}) rectangle (#2-.2,-{#1}+.8);
\draw[pattern=north east lines] (#2-.8,-{#1}) rectangle (#2-.2,-{#1}+.8);
\draw  (#2-.8,-{#1}) rectangle (#2-.2,-{#1}+.8);
}  %red rectangle starting at lasty row arg. 1, wrapping, on col arg.2
\newcommand{\yrec}[2]{\filldraw[yellow]  (#2-.8,-{#1}+.8) rectangle (#2-.2,-{#1}+.2);
\draw (#2-.8,-{#1}+.8) rectangle (#2-.2,-{#1}+.2);
}  %yellow rectangle starting at row arg. 1, col arg.2
\newcommand{\brec}[3]{\filldraw[black]  (#2-.8,-{#1}+.8) rectangle (#2+#3-1.2,-{#1}+.2);
} %yellow rectangle starting at row arg. 1, col arg.2, length arg.3
\title{Rank Polynomials of Fence Posets are Unimodal}

\author{Ezgi Kantarcı Oğuz}
\email{ezgikantarcioguz@gmail.com}
\address{Dept. of Mathematics, 
Boğaziçi University, 
İstanbul}
\author{Mohan Ravichandran}
\email{mohan.ravichandran@boun.edu.tr}
\address{Dept. of Mathematics, 
Boğaziçi University, 
İstanbul}
\begin{document}
\maketitle
\begin{abstract}

We prove a conjecture of  Morier-Genoud and Ovsienko that says that rank polynomials of  the distributive lattices of lower ideals of fence posets are unimodal. We do this by introducing a related class of \emph{circular} fence posets and proving a stronger version of the conjecture due to  McConville, Sagan and Smyth. We show that the rank polynomials of circular fence posets are symmetric and conjecture that unimodality holds except in some particular cases. We also apply the recent work of Elizalde, Plante, Roby and Sagan on rowmotion on fences and show many of their homomesy results hold for the circular case as well. 
 
\end{abstract}
\section{Introduction}
Fence posets are a natural class of posets that appear in the study of cluster algebras, quiver respresentations and other areas of enumerative combinatorics, see \cite{Saganpaper} for an overview. Let $\alpha=(\alpha_1,\alpha_2,\ldots,\alpha_s)$ be a composition of $n$. The fence poset of $\alpha$, denoted $F(\alpha)$ is the poset on  $x_1,x_2,\ldots,x_{n+1}$ with the order relations: 
\begin{equation*}
x_1\preceq x_2 \preceq \cdots\preceq x_{\alpha_1+1}\succeq x_{\alpha_1+2}\succeq \cdots\succeq x_{\alpha_1+\alpha_2+1}\preceq x_{\alpha_1+\alpha_2+2}\preceq\cdots\preceq x_{\alpha_1+\alpha_2+\alpha_3+1}\succeq \cdots
\end{equation*}
The relations describe a poset with $n+1$ nodes, where $n = \alpha_1 + \ldots + \alpha_s$ is the \emph{size} of $\alpha$,  schematically depicted in Figure~\ref{fig:first} below.
 
\begin{figure}[ht]
\begin{tikzpicture}[scale=.8]
\fill(0,0) circle(.1);
\fill(1,1) circle(.1);
\fill(2,2) circle(.1);
\fill(3,3) circle(.1);
\fill(4,4) circle(.1);
\fill(5,3) circle(.1);
\fill(6,2) circle(.1);
\fill(7,1) circle(.1);
\fill(8,2) circle(.1);
\draw (0,0)--(2,2) (3,3)--(4,4)--(5,3) (6,2)--(7,1)--(8,2);
\draw[dotted] (2,2)--(3,3) (5,3)--(6,2) (8,2)--(10,4);
\draw (0,-.5) node{$x_1$};
\draw (1,0.5) node{$x_2$};
\draw (2,1.5) node{$x_3$};
\draw (3,2.5) node{$x_{\alpha_1}$};
\draw (4,3.5) node{$x_{\alpha_1+1}$};
\draw (5,2.5) node{$x_{\alpha_1+2}$};
\draw (6,1.5) node{$x_{\alpha_1+\alpha_2}$};
\draw (7,.5) node{$x_{\alpha_1+\alpha_2+1}$};
\draw (8.5,1.5) node{$x_{\alpha_1+\alpha_2+2}$};
\end{tikzpicture}
\caption{The fence poset $F(\alpha)$}\label{fig:first}
\end{figure}
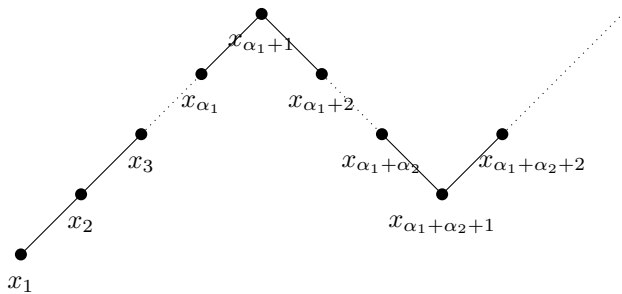

We call the $s$ maximal chains of this poset corresponding to parts of $\alpha$ its \emph{segments}. Lower order ideals of $F(\alpha)$ ordered by inclusion give a distributive lattice which we denote by $J(\alpha)$. The lattice $J(\alpha)$ is ranked by the size of the ideals, with a generating polynomial $R(\alpha;q)= \sum_{I \in J(\alpha)} q^I$, called the \emph{rank polynomial}. We will use $r(\alpha)$ to denote the corresponding \emph{rank sequence} given by the powers of $q$.

\begin{example} \label{ex:2113} The fence poset for $\alpha=(2,1,1,3)$ is given in the left part of Figure \ref{fig:2113}. Note that the ideals of maximal and minimal rank are unique. Ideals of rank $1$ and rank $7$ are given by minima and complements of maxima respectively,  and there are five ideals of rank $2$, depicted in in Figure \ref{fig:2113}, right. The full rank sequence is $(1,3,5,6,6,5,3,2,1)$.
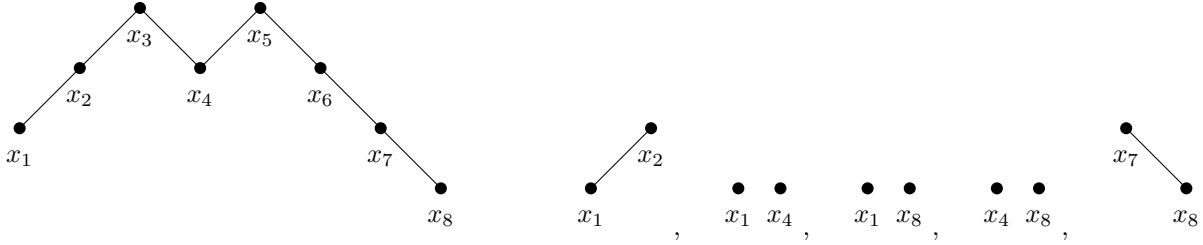
\begin{figure}[ht]\label{fig:2113}
\begin{tikzpicture}[scale=.8]
\fill(0,0) circle(.1);
\fill(1,1) circle(.1);
\fill(2,2) circle(.1);
\fill(3,1) circle(.1);
\fill(4,2) circle(.1);
\fill(5,1) circle(.1);
\fill(6,0) circle(.1);
\fill(7,-1) circle(.1);
\draw (0,0)--(2,2)--(3,1)--(4,2)--(7,-1);
\draw (0,-.5) node{$x_1$};
\draw (1,0.5) node{$x_2$};
\draw (2,1.5) node{$x_3$};
\draw (3,0.5) node{$x_4$};
\draw (4,1.5) node{$x_5$};
\draw (5,0.5) node{$x_6$};
\draw (6,-0.5) node{$x_7$};
\draw (7,-1.5) node{$x_8$};
\end{tikzpicture}\qquad \qquad
\begin{tikzpicture}[scale=.8]
\fill(0,0) circle(.1);
\fill(1,1) circle(.1);
\draw (0,0)--(1,1);
\draw (0,-.5) node{$x_1$};
\draw (1,0.5) node{$x_2$}; 
\end{tikzpicture}, \quad
\begin{tikzpicture}[scale=.8]
\fill(0,0) circle(.1);
\fill(.7,0) circle(.1);
\draw (0,-.5) node{$x_1$};
\draw (.7,-0.5) node{$x_4$}; 
\end{tikzpicture}, \quad
\begin{tikzpicture}[scale=.8]
\fill(0,0) circle(.1);
\fill(.7,0) circle(.1);
\draw (0,-.5) node{$x_1$};
\draw (.7,-0.5) node{$x_8$}; 
\end{tikzpicture}, \quad
\begin{tikzpicture}[scale=.8]
\fill(0,0) circle(.1);
\fill(.7,0) circle(.1);
\draw (0,-.5) node{$x_4$};
\draw (.7,-0.5) node{$x_8$}; 
\end{tikzpicture}, \quad
\begin{tikzpicture}[scale=.8]
\fill(0,1) circle(.1);
\fill(1,0) circle(.1);
\draw (0,1)--(1,0);
\draw (0,.5) node{$x_7$};
\draw (1,-0.5) node{$x_8$}; 
\end{tikzpicture}
\caption{The fence poset $F(2,1,1,3)$ (left) and its five ideals of rank $2$ (right).}
\end{figure}
\end{example}

\begin{comment}
\begin{example}\label{F2113} The fence poset for $\alpha=(2,1,1,3)$ is given in Figure \ref{fig:2113intro}. 
\begin{figure}[ht]\label{fig:2113intro}
\begin{tikzpicture}[scale=.8]
\fill(0,0) circle(.1);
\fill(1,1) circle(.1);
\fill(2,2) circle(.1);
\fill(3,1) circle(.1);
\fill(4,2) circle(.1);
\fill(5,1) circle(.1);
\fill(6,0) circle(.1);
\fill(7,-1) circle(.1);
\draw (0,0)--(2,2)--(3,1)--(4,2)--(7,-1);
\draw (0,-.5) node{$x_1$};
\draw (1,0.5) node{$x_2$};
\draw (2,1.5) node{$x_3$};
\draw (3,0.5) node{$x_4$};
\draw (4,1.5) node{$x_5$};
\draw (5,0.5) node{$x_6$};
\draw (6,-0.5) node{$x_7$};
\draw (7,-1.5) node{$x_8$};
\end{tikzpicture}
\caption{The fence poset $F(2,1,1,3)$}
\end{figure}
The full rank sequence for this poset is $(1,3,5,6,6,5,3,2,1)$, which is unimodal. 
\end{example}
\end{comment}

The rank sequences of fence posets were used by Morier-Genoud and Ovsienko in \cite{originalconj} in their recent work defining $q$-analogues of the rational numbers. Their $q$-rationals are defined by the ratio of the rank polynomials for two compositions given by the continued fraction expression of the rationals considered and enjoy several interesting properties including a type of convergence which allows one to extend their definition to obtain $q$-real numbers. They also proposed the following conjecture in their paper, the proof of which is the main result in this paper. 
\begin{thm}[Conjecture $1.4$ in \cite{originalconj}]\label{thm:main0}
The rank polynomials of fence posets are unimodal. 
\end{thm}

While there was no a priori reason for the authors to expect that this conjecture holds, there was ample numerical evidence. Results predating the conjecture itself were given by Salvi and Munarini \cite{crown}, who considered the case when all parts equal to $1$. Claussen \cite{claussen2020expansion} showed that the conjecture holds when the composition has at most $4$ parts. Further partial progress was made by McConville, Sagan and Smyth \cite{Saganpaper}, who proved the conjecture in the case where the first segment is larger than the sum of the others and proposed the following strengthening of this conjecture. The various interlacing properties referred to in the next theorem are defined in the next section. 

\begin{thm}[Conjecture 1.4 in \cite{Saganpaper}]\label{thm:main} Suppose $\alpha=(\alpha_1,\alpha_2,\ldots,\alpha_s)$.
\begin{enumerate}

\item[(a)]  If $s=1$ then $r(\alpha) = (1,1,\ldots,1)$ is symmetric.
\item[(b)]  If $s$ is even, then $r(\alpha)$ is bottom interlacing.
\item[(c)] If $s\ge 3$ is odd we have:
	\begin{enumerate}
	\item[(i)] If $\alpha_1>\alpha_s$ then $r(\alpha)$ is bottom interlacing.
	\item[(ii)] If $\alpha_1<\alpha_s$ then $r(\alpha)$ is top interlacing.
	\item[(iii)] If $\alpha_1=\alpha_s$ then $r(\alpha)$ is symmetric, bottom interlacing, or top interlacing depending on whether 
	$r(\alpha_2,\alpha_3,\ldots,\alpha_{s-1})$ is symmetric, top interlacing, or bottom interlacing, respectively.
\end{enumerate}
\end{enumerate}
\end{thm}

One of the challenges in proving the above theorem comes from the feature that there are fence posets whose rank sequences can have long flat parts.
\begin{example}
The rank sequence of the composition $(a, 1, 1, 1)$ where $a > 2$ is
\[r(\alpha) = (1, 3, 4, \overbrace{5}^{a-2}, 4, 3, 2, 1).\]
\end{example}
We will describe the main ideas in our proof later but it is noteworthy that our proof is purely combinatorial and essentially constructive, in that we can effectively describe injections that realize the desired unimodality. Unimodality of combinatorial sequences is often deduced by first proving stronger properties of the sequence such as log concavity, ultra log concavity or even real rootedness, but for this problem, none of these stronger properties need hold. Indeed to see that even log concavity need not hold, we see that for the fence poset $F(\alpha) = F(2, 1, 1, 3)$ described in example \ref{ex:2113}, we have  
\[9 = r(\alpha)[6]^2 < r(\alpha)[5]\, r(\alpha)[7] = 5\cdot 2 = 10,\] where, we use the notation $r(\alpha)[k]$ to refer to the number of $k$ ideals of the fence poset $F(\alpha)$.  

They key idea in our proof is to navigate between the properties of fence posets and those of the closely related class of \emph{circular fence posets}. For a composition $\alpha=(\alpha_1,\alpha_2,\ldots,\alpha_{2s})$ of $n$ we define the \emph{circular fence poset} of $\alpha$, denoted $\overline{F}(\alpha)$ as the fence poset of $\alpha$ where $x_{n+1}$ and $x_1$ are taken to be equal, so we get a circular poset with $n$ nodes. 

\begin{example}
The circular fence poset $\bar{F}(2, 1, 1,3)$ is obtained from the regular fence poset $F(2, 1, 1, 3)$ (see figure \ref{fig:2113}) by identifying the vertices $x_1$ and $x_8$, yielding a poset on $7$ elements.  Referring once again to figure \ref{fig:2113}, given that we have identified $x_1$ and $x_8$, two of the five ideals of size two are identical in the circular version and the ideal $(x_1, x_8)$ does not appear. Thus, the number of rank $2$ ideals in $\bar{F}(2, 1, 1,3)$ is $3$. The full rank sequence for $\bar{F}(2, 1, 1,3)$ is $(1, 2, 3, 4, 4, 3, 2, 1)$.
\end{example}

We will use $\bar{J}(\alpha)$ to refer to the lattice of lower ideals of $\bar{F}(\alpha)$,  $\bar{R}(\alpha; q)$ to refer to the rank polynomal of $\bar{J}(\alpha)$ and $\bar{r}(\alpha)$ to refer to the rank sequence. Rank polynomials for circular fence posets behave slightly differently from those for regular fence posets; there are examples where they fail to be unimodal, see section $4$ for a discussion and a characterization. However, they do satisfy a highly convenient property. 

\begin{thm}
Rank polynomials of circular fence posets are \emph{symmetric}. 
\end{thm}

Given a fence poset, there are several naturally related circular fence posets. Our proof consists of relating the rank polynomials of these various posets and  inductively proving a number of ancillary results. One of the byproducts of our proof is the following result that might be of independent interest. 
\begin{thm}
Let $\alpha = (\alpha_1, \ldots, \alpha_{2s})$ be a composition with an even number of parts and consider any cyclic shift of $\alpha$,  $\beta = (\alpha_k, \alpha_{k+1}, \ldots, \alpha_{2s}, \alpha_1, \alpha_2, \ldots, \alpha_{k-1})$. Then 
\[\bar{R}(\alpha;q) = \bar{R}(\beta;q).\]
In other words, the rank polynomial of a circular fence poset is well defined over \emph{circular} compositions. 
\end{thm}

As mentioned above, when it comes to circular fence posets, unimodality need not hold. 
\begin{example}
Let $\alpha = (1, a, 1, a)$ be a composition. A direct calculation shows that the rank sequence is
\[r(\alpha) = (1, 2, \ldots, a, a+1, a, a+1, a, a-1, \ldots, 1).\]
This sequence has a dip in the middle term and is not unimodal. 
\end{example}

%In section $4$, we analyze this phenomenon and show the following. 
%\begin{thm}
%Let $\alpha$ be a composition with an even number of parts. Then provided $\alpha$ is not of the form $(\overbrace{1, a, 1, a, \ldots, 1, a}^{2s})$ for some $a \in \mathbb{N}$, the rank polynomial $\bar{R}(\alpha)$ is unimodal. For these exceptional cases, the polynomial is almost unimodal : There is a dip in the middle term, but apart from this, one has unimodality. 
%\end{thm}

%We then proceed to investigate circular fences in more detail and study the natural combinatorial operations of rowmotion and homemesy on these posets. \todo{Add some material}

%In the last section of this paper, we propose a natural higher dimensional generalization of fence posets, prove some preliminary results and propose a couple of conjectures. 

\section{Notation and Terminology}

%Note that the map $\Sigma$ that takes the complement of an order ideal gives us a bijection between the upper and lower order ideals of $F(\alpha)$ that reverses the rank sequence.  

Let $P$ be a finite poset. A subset $I$ of $P$ is said to be a lower order ideal (resp. upper order ideal) if when $x \in I$, any $y \preceq x$ (resp. any $y \succeq x$) lies in $I$ as well. 
We will use the word ``ideal'' to denote a lower order ideal, unless stated otherwise, and use the notation $I \trianglelefteq P$. Ideals (or upper order ideals) of a poset $P$ ordered by inclusion give the structure of a distributive lattice $J(P)$, ranked by the number of elements. See \cite{Stanley} Chapter 3.4 for a detailed discussion. For the purposes of this work, we will use the work "rank" exclusively to refer to the rank structure of the order ideal lattice. Note that taking the setwise complement of an ideal gives an upper order ideal of complementary rank. 

We will be interested in the case where $P$ is a fence, or a circular fence, and consider the corresponding rank sequence and rank polynomial. The fences are defined to start with an up step, but as flipping a fence vertically only reverses the rank sequence, their structure can be inferred easily. Fences that start with a down step will come up at a few instances in our proofs, but instead of developing a separate notation for upside down fences, we will allow the first part of the composition to be zero in those instances.

A sequence is called \emph{unimodal} if there exists an index $m$ such that  $$
a_0\le a_1 \le \cdots \le a_m \ge a_{m+1}\ge \ldots\ge a_{n}$$. It was conjectured in \cite{originalconj} that the rank sequence of $J(\alpha)$ is unimodal. A more specific conjecture about the behaviour of the coefficients was given in \cite{Saganpaper}. 
A sequence is called {\em top interlacing} if

$$
a_0\le a_n \le a_1\le a_{n-1} \le \ldots\le a_{\ce{n/2}}$$

where $\ce{\cdot}$ is the ceiling function.  Similarly, the sequence is {\em bottom interlacing} if
$$
a_n\le a_0 \le a_{n-1} \le a_1 \le \ldots \le a_{\fl{n/2}}
$$
with $\fl{\cdot}$ being the floor function. Note that top interlacing as well as bottom interlacing sequences are unimodal.

To prove this Theorem \ref{thm:main}, we will define a circular version of the fence poset, where the first and last node are related.

 \section{Circular Fences}  
For a composition $\alpha=(\alpha_1,\alpha_2,\ldots,\alpha_{2s})$ of $n$ we define the \emph{circular fence poset} of $\alpha$, denoted $\overline{F}(\alpha)$ to be the fence poset of $\alpha$ with the additional relation $x_{n+1}=x_1$, so that we end up with a circular poset with $n$ nodes. We will denote the corresponding order ideal lattice, rank polynomial and rank sequence by $\overline{J}(\alpha), \overline{R}(\alpha;q)$ and $\overline{r}(\alpha)$  respectively. We will call the nodes that correspond to $\alpha_i$ the $i$th \emph{segment} of $\overline{F}(\alpha)$

Circular fences have substantial intrinsic symmetry. Shifting the parts cyclically by two steps gives the same object and reversing the order of the parts preserves the rank sequence. In the special case when all the segments are of size $1$, the object we obtain is called a \emph{crown}. Crowns were previously studied in \cite{crown} where it was shown that the corresponding rank polynomials are symmetric, and that they are unimodal when the number of segments is different than $4$. Examining the one step shift allows us to directly say that the symmetry holds when one of the segments is larger as well. This will serve as the basis to prove that in fact, for any circular fence poset we get a rank symmetric lattice.

\begin{lemma} Shifting the parts of $\alpha$ cyclically by one step reverses the rank sequence $\overline{r}(\alpha)$. In particular \label{lemma:basis}$\overline{R}((k,1,1,1,\ldots,1);q)$ where the number of segments is even is symmetric for any $k\in \mathbb{N}$. 
\end{lemma}
\begin{proof} This follows as a cyclic shift of one step on a circular fence is equivalent to reversing the order relation or flipping the poset upside down. Making a cyclic shift of one step followed by reversing the parts of $(k,1,1,1,\ldots,1)$ gives $(1,1,1,1,\ldots,1,k,1)$ which has the exact same structure but a reversed rank sequence.

%The three steps in the proof are outlined from $I-IV$. Figure $I$ shows the circular fence poset $\bar{F}(3, 1,1 1)$ and a lower ideal marked in red. Figure $II$ shows the complement of this ideal which is now an upper ideal. Figure $III$ shows the vertically reflected poset, which is the same as $\bar{F}(1, 3, 1, 1)$ and the vertical reflected subset which is now a lower ideal. Figure $IV$ is the horizontal reflection of figure $III$. Ths subset is still a lower ideal and the poset is another representation of $\bar{F}(3, 1, 1, 1)$.
%Lower ideal posets and upper ideal posets are in bijection via taking setwise complement. Reflecting the upper ideal posets gives us a bijection between ideals of rank $r$ with ideals of length $n-r$ where $n$ is the size of the permutation.
\end{proof} 

In general, rank polynomials for circular fences are no easier to calculate than their non circular counterparts and we only have formulas for a limited number of cases. The case when $\alpha=(1,a,1,a,\ldots,1,a)$ was considered in \cite{crown2}. They were able to formulate the rank polynomial in terms of Chebyshev polynomials of the first kind, defined recursively by $T_0(q)=1$, $T_1(q)=q$ and $T_{n+2}(q) = 2q \,T_{n+1}(q) - T_n(q)$. 

\begin{prop}[\cite{crown2}] We have
$$\displaystyle \overline{R}((1,a,1,a,\ldots,1,a)=2q^{({(a-1)s})/{2}}\,{T}_{a-1}\left(\frac{1+q+q^2+\cdots+q^s}{2q^{{s}/{2}}}\right)$$
where $2s$ is the number of segments of $(1,a,1,a,\ldots,1,a)$.
\end{prop} 

Note that when $s=2$, we get the polynomial $1+2q+3q^2+\cdots+(a+1)q^a+aq^{a+1}+(a+1)q^{a+2}+(a)q^{a+3}+\cdots+2q^{2a+1}+q^{2a+2}$, which is not unimodal. 

Some other small cases that can be easily calculated by hand are listed in Table \ref{tab:smallcases} below.
\begin{table}[ht]
    \centering
\begin{tabular}{||c c  c||} 
 \hline 
 $\alpha$ & Ideal Count & Rank Polynomial \\ [0.5ex] 
 \hline\hline  
 $(a,b)$ & $ab+2$ & $1+q[a]_q[b]_q+q^{a+b}$  \\ [0.5ex] 
 \hline 
 $(a,1,b,1)$ & $ab+2a+2b+2$ & $[a+2]_q[b+2]_q-q^{a+1}-q^{b+1}$  \\[0.5ex] 
 \hline 
 $(a,b,c,d)$ & $abcd+ab+cd+ad+bc+2$ & \begin{tabular}{c}
      $1+q[a]_q[d]_q+q[b]_q[c]_q+q^{a+b+1}[c]_q[d]_q$ \\
     $ +q^{c+d+1}[a]_q[b]_q+q^{a+b+c+d}$ 
 \end{tabular} \\
 \hline
 $(a,a,a,a)$ & $a^4+4a^2+2$ & $1+([a]_q)^4+(2q^{2a+1}+2q)([a]_q)^2+q^{4a}$  \\[0.5ex] 
 \hline
\end{tabular}

\vspace{2mm}
    \caption{Ideal count and rank polynomial for small cases}
    \label{tab:smallcases}
\end{table}

The cases of $(a,b)$ and $(1,a,1,b)$ are indeed quite straightforward. The lattice formed by the ideals of $\overline{F}(a,b)$ is formed by the direct product of two chains of lengths $a$ and $b$, with an added minimum element (for the empty ideal) and maximum element (for the full ideal): $\hat{0} \oplus C_{a}\times C_{b} \oplus \hat{1}$. Here, the position on $C_{a}$ corresponds to the number of unshared elements in the segment of size $a$, whereas the position on $C_{b}$ describes the number of unshared elements in the segment of size $b$. The natural symmetric chain decomposition on $C_{a}\times C_{b}$ can easily be extended to accommodate the two added nodes, as seen in Figure \ref{fig:48latticechains} for the example of $(5,8)$. We get $ab+2$ ideals with the corresponding rank polynomial $\overline{R}((a,b);q)=1+q[a]_q[b]_q+q^{a+b}$.

\begin{figure}[ht]
    \centering
    \include{symmetricdecompositionfigure}
    \caption{The lattice $J((5,8))$ (left) has a natural symmetric chain decomposition (right)}
    \label{fig:48latticechains}
\end{figure}

When we have $(1,a,1,b)$, any ideal of size $k$ is a partitioning of $k$ into two parts $p_1\leq a$ and $p_2\leq b $ such that $p_1=a \Rightarrow p_2 \neq 0$ and $p_2=b \Rightarrow p_1\neq 0$. The lattice we obtained can be visualised as $C_{a+1} \times C_{b+1}$ with the two opposite corners deleted. When $a\neq b$ this also has a natural symmetric chain decomposition. When $a=b$ however, we have no such decomposition as the resulting rank polynomial is not unimodal, see Figure \ref{fig:1417}. We have $(a+1)(b+1)-2$ ideals, with $\overline{R}((1,a,1,b);q)=[a+2]_q[b+2]_q-q^{a+1}-q^{b+1}$.
\begin{figure}[ht]
    \centering
    \begin{tikzpicture}[scale=.6]
 \begin{scope} [xshift=15 ,yshift=0, rotate=45]
\draw (0,0)grid (7,4);
\draw[ultra thick,white] (1,4)--(0,4)--(0,3) (6,0)--(7,0)--(7,1);
\foreach \x in {0,1,2,3,4,5,6,7}
   \foreach \y in {0,1,2,3,4}{
\fill(\x,\y) circle(.1);}
\fill[white](0,4) circle(.2);
\fill[white](7,0) circle(.2);
\end{scope}
\end{tikzpicture} \quad
    \begin{tikzpicture}[scale=.6]
 \begin{scope} [xshift=15 ,yshift=0, rotate=45]
\draw[gray](0,0)grid (7,4);
\draw[ultra thick,white] (1,4)--(0,4)--(0,3) (6,0)--(7,0)--(7,1);
\foreach \x in {0,1,2,3,4,5,6,7}
   \foreach \y in {0,1,2,3,4}{
\fill[red](\x,\y) circle(.1);}
\fill[white](0,4) circle(.2);
\fill[white](7,0) circle(.2);
\draw[ultragreen, dashed, ultra thick] (0.3,5.3)--(6.3,-.8);
\draw[ultra thick,red](0,0)--(6,0)--(6,1)--(7,1)--(7,4) (0,1)--(5,1)--(5,2)--(6,2)--(6,4) (0,2)--(4,2)--(4,3)--(5,3)--(5,4) (0,3)--(3,3)--(3,4)--(4,4) (1,4)--(2,4);
\end{scope}
\end{tikzpicture} \qquad
    \begin{tikzpicture}[scale=.6]
 \begin{scope} [xshift=15 ,yshift=0, rotate=45]
\draw (0,0)grid (5,5);
\draw[ultra thick,white] (1,5)--(0,5)--(0,4) (4,0)--(5,0)--(5,1);
\foreach \x in {0,1,2,3,4,5}
   \foreach \y in {0,1,2,3,4,5}{
\fill(\x,\y) circle(.1);}
\fill[white](0,5) circle(.2);
\fill[white](5,0) circle(.2);
\draw[ultragreen, dashed, ultra thick] (-.2,5.2)--(5.2,-.2);
\end{scope}
\end{tikzpicture}
    \caption{The lattice $J((1,3,1,6))$ (left) has a natural symmetric chain decomposition (middle) whereas $J((1,4,1,4))$ (right) can not be decomposed into symmetric chains.}
    \label{fig:1417}
\end{figure}
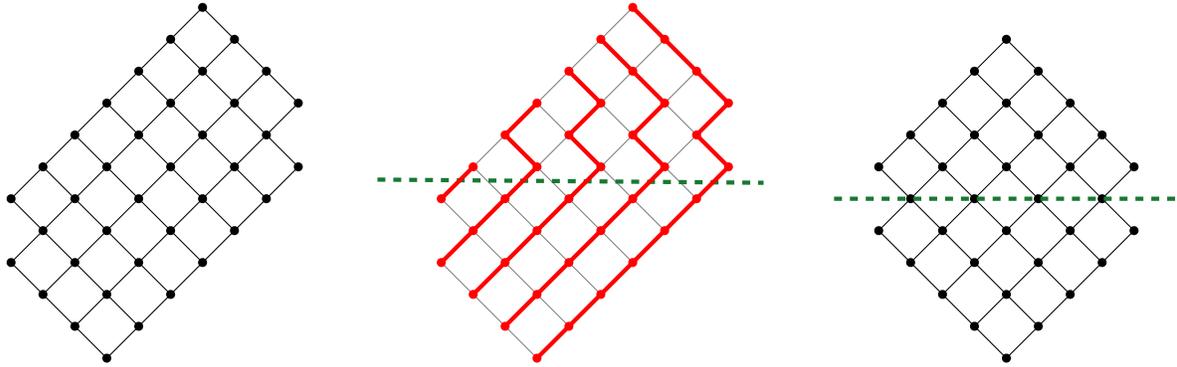
\renewcommand{\arraystretch}{1.5}

\section{An Example} \label{sec:example}

In this chapter, we will consider ways of closing up a fence poset to get a rounded fence through the example of $\alpha=(2,1,1,3)$. The ideas illustrated here will be the backbone of the proofs that will be given in the upcoming sections.

\begin{method}Letting $x_1=x_8$. \end{method}
\begin{figure}[ht]
\begin{tikzpicture}[scale=.8]
\fill[red](3.5,-1) circle(.1);
\fill(1,1) circle(.1);
\fill(2,2) circle(.1);
\fill(3,1) circle(.1);
\fill(4,2) circle(.1);
\fill(5,1) circle(.1);
\fill(6,0) circle(.1);
\draw (1,1)--(2,2)--(3,1)--(4,2)--(6,0);
\draw[red,thick] (1,1)--(3.5,-1)--(6,0);
\draw (3.5,-1.5) node{$x_1=x_8$};
\draw (1,0.5) node{$x_2$};
\draw (2,1.5) node{$x_3$};
\draw (3,1.5) node{$x_4$};
\draw (4,1.5) node{$x_5$};
\draw (5,0.5) node{$x_6$};
\draw (6,-0.5) node{$x_7$};
\end{tikzpicture} \qquad \raisebox{1.3cm}{$\Longleftrightarrow$}\qquad \raisebox{.2cm}{\begin{tikzpicture}[scale=.8]
\fill(1,1) circle(.1);
\fill(2,2) circle(.1);
\fill(3,1) circle(.1);
\fill(4,2) circle(.1);
\fill(5,1) circle(.1);
\fill(6,0) circle(.1);
\draw (1,1)--(2,2)--(3,1)--(4,2)--(6,0);
\draw (1,0.5) node{$x_2$};
\draw (2,1.5) node{$x_3$};
\draw (3,0.5) node{$x_4$};
\draw (4,1.5) node{$x_5$};
\draw (5,0.5) node{$x_6$};
\draw (6,-0.5) node{$x_7$};
\end{tikzpicture}}
    \caption{The ideals of $\overline{F}(2,1,1,3)$ that contain $x_1$ $\Longleftrightarrow$ Ideals of $F(1,1,1,2)$}
\end{figure}
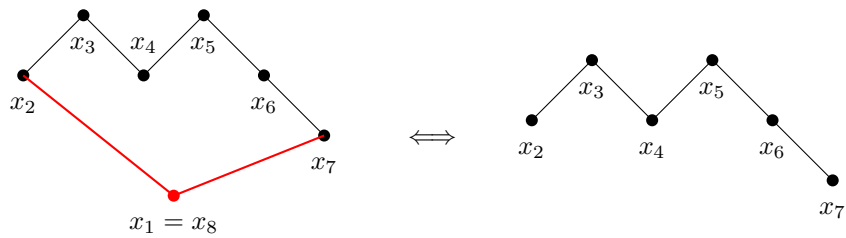

This natural choice of setting $x_0=x_8$ gives us the rounded fence poset $\overline{F}(\alpha)$, which has the disadvantage of having only $7$ nodes, so that we do not have all the structure of our original poset included in this circular version. In particular, we lose the ideals that contain only one of $x_1$ and $x_8$.  

For the ideals of $\alpha$ that contain $x_1$ but not $x_8$, any node above $x_8$ is also not included, and there is no effect on nodes above $x_1$, so we get a bijection with ideals of $F(1.1)$ as depicted in Figure \ref{fig:2113left} below.
\begin{figure}[ht] 
\begin{tikzpicture}[scale=.8]
\fill[red](0,0) circle(.1);
\fill(1,1) circle(.1);
\fill(2,2) circle(.1);
\fill(3,1) circle(.1);
\fill[red](4,2) circle(.1);
\fill[red](5,1) circle(.1);
\fill[red](6,0) circle(.1);
\fill[red](7,-1) circle(.1);
\draw (0,0)--(2,2)--(3,1)--(4,2)--(7,-1);
\draw[red,thick] (4,2)--(7,-1);
\draw (0,-.5) node{$x_1$};
\draw (1,0.5) node{$x_2$};
\draw (2,1.5) node{$x_3$};
\draw (3,1.5) node{$x_4$};
\draw (4,1.5) node{$x_5$};
\draw (5,0.5) node{$x_6$};
\draw (6,0.5) node{$x_7$};
\draw (7,-.5) node{$x_8$};
\end{tikzpicture}\qquad \raisebox{1.3cm}{$\Longleftrightarrow$}\qquad \raisebox{.2cm}{\begin{tikzpicture}[scale=.8]
\fill(1,1) circle(.1);
\fill(2,2) circle(.1);
\fill(3,1) circle(.1);
\draw (1,1)--(2,2)--(3,1);
\draw (1,0.5) node{$x_2$};
\draw (2,1.5) node{$x_3$};
\draw (3,0.5) node{$x_4$};
\end{tikzpicture}}
    \centering
    \caption{The ideals of ${F}(2,1,1,3)$ that contain $x_1$ but not $x_8$ $\Longleftrightarrow$ Ideals of $F(1,1)$}
    \label{fig:2113left}
\end{figure}
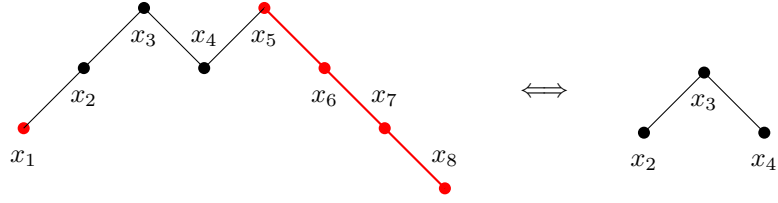

Similarly, ideals that contain $x_8$ but not $x_1$ are in bijection with ideals of $1,2$, see Figure \ref{fig:2113right}.

\begin{figure}[ht]\label{fig:2113right}
\begin{tikzpicture}[scale=.8]
\fill[red](0,0) circle(.1);
\fill[red](1,1) circle(.1);
\fill(2,2) circle(.1);
\fill(3,1) circle(.1);
\fill(4,2) circle(.1);
\fill(5,1) circle(.1);
\fill(6,0) circle(.1);
\fill[red](7,-1) circle(.1);
\draw (0,0)--(2,2)--(3,1)--(4,2)--(7,-1);
\draw[red,thick] (2,2)--(0,0);
\draw (0,-.5) node{$x_1$};
\draw (1,0.5) node{$x_2$};
\draw (2,1.5) node{$x_3$};
\draw (3,0.5) node{$x_4$};
\draw (4,1.5) node{$x_5$};
\draw (5,0.5) node{$x_6$};
\draw (6,-0.5) node{$x_7$};
\draw (7,-1.5) node{$x_8$};
\end{tikzpicture}\qquad \raisebox{1.7cm}{$\Longleftrightarrow$}\qquad \raisebox{.7cm}{\begin{tikzpicture}[scale=.8]
\fill(3,1) circle(.1);
\fill(4,2) circle(.1);
\fill(5,1) circle(.1);
\fill(6,0) circle(.1);
\draw (3,1)--(4,2)--(6,0);
\draw (3,0.5) node{$x_4$};
\draw (4,1.5) node{$x_5$};
\draw (5,0.5) node{$x_6$};
\draw (6,-0.5) node{$x_7$};
\end{tikzpicture}}
    \centering
    \caption{The ideals of ${F}(3,1,1,4)$ that contain $x_8$ but not $x_1$ $\Longleftrightarrow$ Ideals of $F(1,2)$}
\end{figure}
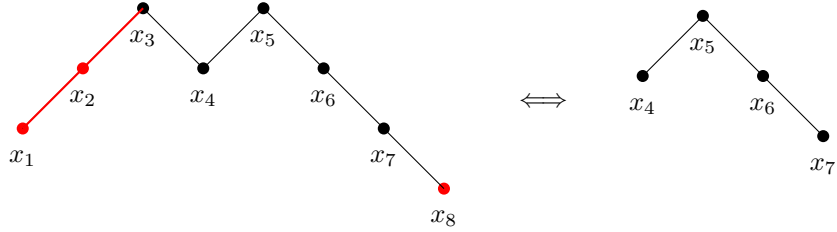

The connection between the rank polynomials, consequentally, is a bit tricky. The ideals of $F(2,1,1,3)$ that do not contain $x_1$ or $x_8$ do not contain any node above them, so we have only two such ideals, the empty one and the one that consists of just $x_4$. Subtracting all these gives us the contribution of the ideals that contain both $x_0$ and $x_8$, which can also be calculated via adding $2$ nodes to each ideal of $F(1,1,1,2)$.
\begin{eqnarray*}
q^2 R(1,1,1,2)&=& R(2,1,1,3)-q R(1,1) - q R(1,2) - (1+q).\\
&=& (1+3q+5q^2+ 6q^3 + 6q^4 + 5q^5 + 3q^6 + 2q^7 + q^8)-q(1+2q+q^2+q^3)\\&&-q(1+2q+2q^2+q^3+q^4)-(1+q)\\
&=& q^2+3q^3+4q^4+4q^5+3q^6+2q^7+q^8.
\end{eqnarray*} 
These ideals are shifted by $q^{-1}$ to give the ideals of $\overline{F}(2,1,1,3)$ that contain $x_1=x_8$. The two that do not contribute $1+q$, so that we get the following rank symmetric polynomial:
\begin{eqnarray*}
\overline{R}(2,1,1,3) &=& (q^{-1})(q^2+3q^3+4q^4+4q^5+3q^6+2q^7+q^8)+1+q\\
&=& 1+2q+3q^2+4q^3+4q^4+3q^5+2q^6+q^7.
\end{eqnarray*} 

Adding the relation that $x_1$ is above (or below) $x_8$ allows us to get a circular fence with the same number of nodes.

\begin{method} \label{method:connect} Connecting $x_1$ and $x_8$.\end{method}

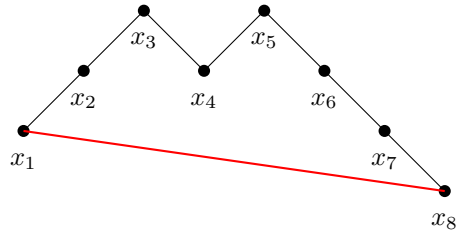
\begin{figure}[ht]
\begin{tikzpicture}[scale=.8]
\fill(0,0) circle(.1);
\fill(1,1) circle(.1);
\fill(2,2) circle(.1);
\fill(3,1) circle(.1);
\fill(4,2) circle(.1);
\fill(5,1) circle(.1);
\fill(6,0) circle(.1);
\fill(7,-1) circle(.1);
\draw (0,0)--(2,2)--(3,1)--(4,2)--(7,-1);
\draw[red,thick] (0,0)--(7,-1);
\draw (0,-.5) node{$x_1$};
\draw (1,0.5) node{$x_2$};
\draw (2,1.5) node{$x_3$};
\draw (3,0.5) node{$x_4$};
\draw (4,1.5) node{$x_5$};
\draw (5,0.5) node{$x_6$};
\draw (6,-0.5) node{$x_7$};
\draw (7,-1.5) node{$x_8$};
\end{tikzpicture}
    \centering
    \caption{The circular fence $R(3,1,1,3)$ given by assuming $x_1$ is above $x_8$}
\end{figure}

Note that the ideals of $\overline{F}(3,1,1,3)$ give all ideals $I$ of $F(2,1,1,3)$ satisfying $x_1 \in I \Rightarrow x_8 \in I$. The ones that are left over are exactly the ones that contain $x_1$ but not $x_8$ that correspond to the ideals of $F(1,1)$ as we discussed above (See Figure \ref{fig:2113left}). 
The corresponding rank polynomials are also related:
\begin{eqnarray*}
R(2,1,1,3)&=& \overline{R}(3,1,1,3)+ qR(1,1).\\
&=& (1+2q+3q^2+ 5q^3 + 5q^4 + 5q^5 + 3q^6 + 2q^7 + q^8)+q(1+2q+q^2+q^3)\\
&=& 1+3q+5q^2+ 6q^3 + 6q^4 + 5q^5 + 3q^6 + 2q^7 + q^8.
\end{eqnarray*}

Alternatively we can add a new node $x_0$ to complete the cycle.

\begin{method}Adding a new $x_0$ above $x_1$ and $x_8$. \end{method}
\begin{figure}[ht]
\begin{tikzpicture}[scale=.8]
\fill[red](0,0) circle(.1);
\fill(1,1) circle(.1);
\fill(2,2) circle(.1);
\fill(3,1) circle(.1);
\fill(4,2) circle(.1);
\fill(5,1) circle(.1);
\fill(6,0) circle(.1);
\fill[red](7,-1) circle(.1);
\fill[red](3.5,.5) circle(.1);
\draw (0,0)--(2,2)--(3,1)--(4,2)--(7,-1);
\draw[red,thick] (0,0)--(3.5,.5)--(7,-1);
\draw (0,-.5) node{$x_1$};
\draw (1,0.5) node{$x_2$};
\draw (2,1.5) node{$x_3$};
\draw (3,1.5) node{$x_4$};
\draw (4,1.5) node{$x_5$};
\draw (5,0.5) node{$x_6$};
\draw (6,0.5) node{$x_7$};
\draw (7,-.5) node{$x_8$};
\draw (3.5,0) node{$x_0$};
\end{tikzpicture}\qquad \raisebox{1.3cm}{$\Longleftrightarrow$}\qquad \raisebox{.2cm}{\begin{tikzpicture}[scale=.8]
\fill(1,1) circle(.1);
\fill(2,2) circle(.1);
\fill(3,1) circle(.1);
\fill(4,2) circle(.1);
\fill(5,1) circle(.1);
\fill(6,0) circle(.1);
\draw (1,1)--(2,2)--(3,1)--(4,2)--(6,0);
\draw (1,0.5) node{$x_2$};
\draw (2,1.5) node{$x_3$};
\draw (3,0.5) node{$x_4$};
\draw (4,1.5) node{$x_5$};
\draw (5,0.5) node{$x_6$};
\draw (6,-0.5) node{$x_7$};
\end{tikzpicture}}
    \centering
    \caption{The ideals of $\overline{F}(2,1,1,3,1,1)$ that contain $x_0$ $\Longleftrightarrow$ Ideals of $F(1,1,1,2)$}
\end{figure}
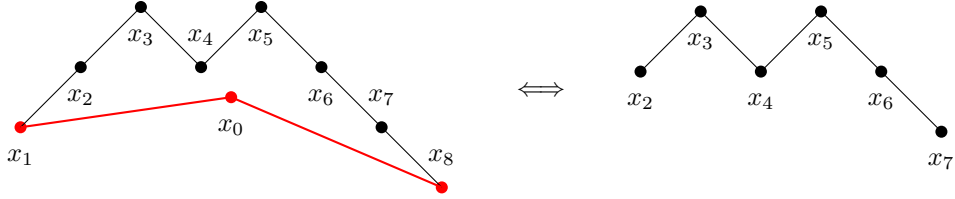

Adding $x_0$ gives us the rounded fence poset $P_T:=\overline{F}(2,1,1,3,1,1)$ with 9 nodes. Ideals of our original poset are exactly the ideals of $P_T$ that do not contain $x_0$. Any ideal that contains $x_0$ also contains $x_1$ and $x_3$ but puts no other restrictions on the inclusion of the other nodes, so these ideals are in bijection with those of $F(1,1,1,2)$. On the rank polynomials side, we get the identity:

\begin{eqnarray*}
&&R(2,1,1,3)= \overline{R}(2,1,1,3,1,1)-q^3 R(1,1,1,2)\\
&&= (1+3q+5q^2+ 7q^3 + 9q^4 + 9q^5 + 7q^6 + 5q^7 + 3q^8+q^9 )-q^3(1+3q+4q^2+4q^3+3q^4+2q^5+q^6)\\
&&= 1+3q+5q^2+ 6q^3 + 6q^4 + 5q^5 + 3q^6 + 2q^7 + q^8.
\end{eqnarray*} 

\begin{method}Adding a new $x_0$ below $x_1$ and $x_8$.\end{method}

\begin{figure}[ht]
\begin{tikzpicture}[scale=.8]
\fill[red](3.5,-2) circle(.1);
\fill[red](0,0) circle(.1);
\fill[red](1,1) circle(.1);
\fill[red](2,2) circle(.1);
\fill(3,1) circle(.1);
\fill[red](4,2) circle(.1);
\fill[red](5,1) circle(.1);
\fill[red](6,0) circle(.1);
\fill[red](7,-1) circle(.1);
\draw (0,0)--(2,2)--(3,1)--(4,2)--(7,-1);
\draw[red,thick] (2,2)--(0,0)--(3.5,-2)--(7,-1)--(4,2);
\draw (0,-.5) node{$x_1$};
\draw (1,0.5) node{$x_2$};
\draw (2,1.5) node{$x_3$};
\draw (3,0.5) node{$x_4$};
\draw (4,1.5) node{$x_5$};
\draw (5,0.5) node{$x_6$};
\draw (6,-0.5) node{$x_7$};
\draw (7,-1.5) node{$x_8$};
\draw (3.5,-2.5) node{$x_0$};
\end{tikzpicture}\qquad \raisebox{1.7cm}{$\Longleftrightarrow$}\qquad \raisebox{1.3cm}{\begin{tikzpicture}[scale=.8]
\fill(3,1) circle(.1);
\draw (3,0.5) node{$x_4$};
\end{tikzpicture}}
    \centering
    \caption{The ideals of $\overline{F}(3,1,1,4)$ that do not contain $x_0$ $\Longleftrightarrow$ Ideals of $F(1)$}
\end{figure}
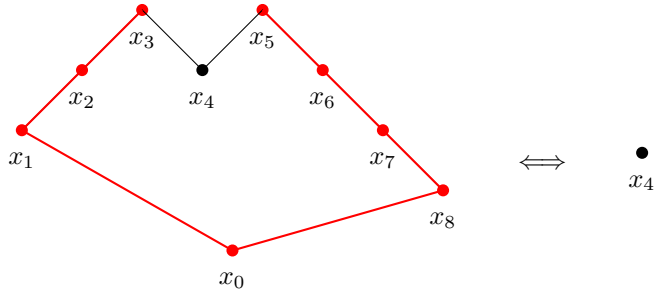

Adding $x_0$ below gives us the rounded fence poset $P_B:=\overline{F}(3,1,1,4)$ with 9 nodes. Ideals that contain $x_0$ are in bijection with our original poset. Any ideal that does not contain $x_0$ can not contain anything above either, and there are only two such ideals, the empty ideal and the rank $1$ ideal that only contains $x_4$:

\begin{eqnarray*}
R(2,1,1,3)&=& (q^{-1})( \overline{R}(3,1,1,4)- R(1)).\\
&=& (q^{-1}) ((1+2q+3q^2+ 5q^3 + 6q^4 + 6q^5 + 5q^6 + 3q^7 + 2q^8+q^9 )-(1+q))\\
&=& 1+3q+5q^2+ 6q^3 + 6q^4 + 5q^5 + 3q^6 + 2q^7 + q^8 .
\end{eqnarray*}

\section{Rank Symmetry in Circular Fences}
In this section, we will prove that the rank polynomial for circular fences is always symmetric. 

\begin{thm} \label{thm:sym} For any composition $\alpha$ of $n$ with an even number of segments, the rank polynomial of $\alpha$ is symmetric with center of symmetry at $n/2$.
\end{thm}

As we already showed symmetry holds in the case $(k,1,1,\ldots,1)$ for any $k$, for any given number of beads we have a case where we already know the rank polynomial is symmetric, so it suffices to show that moving beads around does not break the symmetry. The constructions given in the previous section will be our main tool, with which we will go back and forth between the circular and the non-circular cases. Consider the following statements:

\begin{itemize}  \setlength\itemsep{3mm}
  \item[  ] 
    \item[ \textbf{A(n):}] Given a composition $\beta=(\beta_1,\beta_2,\ldots,\beta_{2s})$ of $n-1$, let $\mathfrak{I}_L$ be the set of ideals of $F(\alpha)$ that include the leftmost node $x_1$, but not the rightmost node $x_n$. Similarly let $\mathfrak{I}_R$ be the set of ideals of $F(\alpha)$ that include the rightmost node but not the leftmost. The polynomial $$ \displaystyle \sum_{I\in \mathfrak{I}_L }q^{|I|}-\sum_{J\in \mathfrak{I}_R }q^{|J|}$$ is symmetric with center of symmetry $n/2$.
    
  \item[ \textbf{B(n):}] Given a composition $\beta=(\beta_1,\beta_2,\ldots,\beta_{2s})$ of $n-1$, where $\beta_1$ and $\beta_{2s}$ are allowed to be $0$ with the convention that when $\beta_1$ is $0$, we get a fence that starts with a down step instead of an up step. Then \[R((\beta_1+1,\beta_2,\ldots,\beta_{2s});q)-R((\beta_1,\beta_2,\ldots,\beta_{2s}+1);q)\] is symmetric around $n/2$.
  \item [\textbf{C(n):}]  Given a composition $\beta=(\beta_1,\beta_2,\ldots,\beta_{2s})$ of $n$, the following difference is symmetric around $(n+1)/2$
  \[\bar{R}((\beta_1+1,\beta_2,\ldots,\beta_{2s});q)-\bar{R}((\beta_1,\beta_2,\ldots,\beta_{2s}+1);q).\]
  \item[   \textbf{D(n):}] Given a composition $\beta = (\beta_1, \ldots, \beta_{2s})$ of $n$, the rank polynomial of the associated circular fence poset $\bar{R}(\beta)$ is symmetric.  
\end{itemize}

We will prove the rank symmetry of circular fences by showing that $\mathbf{D(n)} \Rightarrow \mathbf{A(n)} \Rightarrow \mathbf{B(n)} \Rightarrow \mathbf{C(n+1)}$  which in turn implies $\mathbf{D(n+2)}$. Note that as a byproduct we get that the statements $\mathbf{A(n)}$ and $\mathbf{B(n)}$ about the structure of non-circular fences.

%We will show that the symmetry of circular fences with up to $M$ nodes implies $\bigcup_{i\leq M} A(n)$, which in term implies $\bigcup_{i\leq M} B(n)$. Then we will use $B(M-2)$ to prove the symmetry for circular fences with $M+1$ nodes. 

\begin{proof}[Proof of Theorem \ref{thm:sym}] We will use induction on the size of the composition. If $\alpha$ is a composition of $\leq 3$ with an even number of parts we have only three choices, each of which giving us symmetric rank polynomials:
\[\overline{R}((1,1);q)=1+q+q^2, \qquad \overline{R}((2,1);q)=\overline{R}((1,2);q)=1+q+q^2 + q^3.\]
 Now, let us assume that $\mathbf{D(n)}$ holds, that is, for any composition $\alpha = (\alpha_1, \ldots, \alpha_{2s})$ of $n$, the rank polynomial $\bar{R}(\alpha)$ is symmetric.

 %Assume that symmetry holds when we have $n\leq M$ nodes. We  Let $\alpha=(\alpha_1,\alpha_2,\ldots,\alpha_{2s})$ be a composition with $M$ nodes. We will show that if the rank polynomial of  $\alpha_L=(\alpha_1+1,\alpha_2,\ldots,\alpha_{2s})$ is symmetric with center of symmetry given by $(M+1)/2$, then so is the one for $\alpha_R=(\alpha_1,\alpha_2,\ldots,\alpha_{2s}+1)$. Equivalently, moving nodes from one side of a minimal point to the other preserves symmetry. This would also imply we can move nodes over maximal points, as turning the fence upside down preserves symmetric rank sequences.
%By Lemma \ref{lemma:basis} the result follows. 

\begin{claim}\label{claim1} $\mathbf{A(n)}$ holds.
\end{claim}

\begin{claimproof} We will consider two natural circular fences related to $\beta$: $\overline{R}((\beta_1+1,\beta_2,\ldots,\beta_{2s}))$ given by adding the relation $x_{n}\preceq x_1$ to the fence of $\beta$ and $\overline{R}((\beta_1,\beta_2,\ldots,\beta_{2s}+1))$ given by adding the relation $x_{n}\succeq x_1$ to the fence of $\beta$. Let us denote their rank polynomials by $\overline{\mathfrak{R}}_L(q)$ and $\overline{\mathfrak{R}}_R(q)$ respectively.

Note that
\begin{eqnarray*}
R(\beta;q)&=&\overline{\mathfrak{R}}_L(q)+\sum_{\substack{I\trianglelefteq F(\beta)\\x_{n} \in I,\, x_1 \notin I} }q^{|I|}=\overline{\mathfrak{R}}_R(q)+\sum_{\substack{J\trianglelefteq F(\beta)\\x_{1} \in J,\, x_{n} \notin J} }q^{|J|}
\end{eqnarray*}
Consequently
\begin{eqnarray*}
\sum_{I\in \mathfrak{I}_L }q^{|I|}-\sum_{J\in \mathfrak{I}_R }q^{|J|}&=&\overline{\mathfrak{R}}_R(q)-\overline{\mathfrak{R}}_L(q).
\end{eqnarray*}
Both rank polynomials belong to circular fences with $n$ nodes, which are symmetric around $n/2$ by our hypothesis.
\end{claimproof}

\begin{claim} \label{claim2} $\mathbf{B(n)}$ holds.

\end{claim}
\begin{claimproof} Let $\mathfrak{F}_L=F((\beta_1+1,\beta_2,\ldots,\beta_{2s}))$ be the fence with $n+1$ nodes given by adding a new node to the (possibly empty) leftmost segment of $\beta$ by $1$. Similarly, let $\mathfrak{F}_R=F((\beta_1,\beta_2,\ldots,\beta_{2s}+1))$.

We want to show that the following polynomial is symmetric around $n/2$:
 $$ \displaystyle \sum_{I \trianglelefteq \mathfrak{F}_L }q^{|I|}-\sum_{J \trianglelefteq \mathfrak{F}_R }q^{|J|}.$$
We will make use of the circular fence $\overline{\mathfrak{RF}}$ for $(\beta_1+1,\beta_2,\ldots,\beta_{2s}+1)$. Note that we can obtain $\overline{\mathfrak{RF}}$ from $\mathfrak{F}_L$ by adding the relation $x_1 \preceq x_{n+1}$ (see Method~\ref{method:connect} from Section~\ref{sec:example} for an example) so that:
$$\displaystyle \sum_{I\trianglelefteq \overline{\mathfrak{RF}} } q^{|I|} =\sum_{I\trianglelefteq \mathfrak{F}_L }q^{|I|}-\sum_{\substack{I\trianglelefteq \mathfrak{F}_L\\x_{n+1} \in I,\, x_1 \notin I} }q^{|I|}.$$ 
Similarly we have:
$$\displaystyle \sum_{I\trianglelefteq \overline{\mathfrak{RF}} } q^{|I|} =\sum_{J\trianglelefteq \mathfrak{F}_R }q^{|J|}-\sum_{\substack{J\trianglelefteq \mathfrak{F}_R\\x_1 \in J,\, x_{n+1} \notin J} }q^{|J|}.$$ 
This yields 
%Now we will try to describe the difference in terms of the ideals of $F(\beta)$:
\begin{eqnarray*}
\displaystyle \sum_{I \trianglelefteq \mathfrak{F}_L }q^{|I|}-\sum_{J \trianglelefteq \mathfrak{F}_R }q^{|J|} &=&\sum_{\substack{J\trianglelefteq \mathfrak{F}_R\\x_1 \in J,\, x_{n+1} \notin J} }q^{|J|}-\sum_{\substack{I\trianglelefteq \mathfrak{F}_L\\x_{n+1} \in I,\, x_1 \notin I} }q^{|I|}
\end{eqnarray*}
We observe that 
\[\sum_{\substack{I\trianglelefteq \mathfrak{F}_L\\x_{n+1} \in I,\, x_1 \notin I} }q^{|I|} = \sum_{\substack{I\trianglelefteq F(\beta)\\x_{n} \in I,\, x_1 \notin I} }q^{|I|}. \qquad \sum_{\substack{J\trianglelefteq \mathfrak{F}_R\\x_1 \in J,\, x_{n+1} \notin J} }q^{|J|} =  \sum_{\substack{J\trianglelefteq F(\beta)\\x_1 \in J,\, x_{n} \notin J} }q^{|J|}.\]
This is simply because if $x_1 \not\in I \trianglelefteq \mathfrak{F}_L$, then no nodes from the first segment can be in the ideal and this sets up a bijection between the two sets of ideals in the left equation above.  The second equation may be similarly justified. We conclude that 
\begin{eqnarray*}
 \sum_{\substack{J\trianglelefteq \mathfrak{F}_R\\x_1 \in J,\, x_{n+1} \notin J} }q^{|J|}-\sum_{\substack{I\trianglelefteq \mathfrak{F}_L\\x_{n+1} \in I,\, x_1 \notin I} }q^{|I|} = \sum_{\substack{J\trianglelefteq F(\beta)\\x_1 \in J,\, x_{n} \notin J} }q^{|J|}-\sum_{\substack{I\trianglelefteq F(\beta)\\x_{n} \in I,\, x_1 \notin I} }q^{|I|}.
\end{eqnarray*}

By Claim 1, this difference is symmetric with center of symmetry at $n/2$. 
\end{claimproof}

\begin{claim} \label{claim3} $\mathbf{C(n+1)}$ holds.

\end{claim}
\begin{claimproof} 
Let $\alpha = (\alpha_1, \ldots, \alpha_{2s})$ be a composition of $n+1$ and let $\alpha_L = (\alpha_1+1, \ldots, \alpha_{2s})$ and $\alpha_R = (\alpha_1, \ldots, \alpha_{2s}+1)$. Let  $\mathfrak{I}_L$ be the set of ideals of $\overline{F}(\alpha_L)$ and $\mathfrak{I}_R$ be the set of ideals of $\overline{F}(\alpha_R)$. The ideals that do not contain $x_1=x_{n+3}$ are in bijection as they contain no nodes from the first or last segments. That means we can limit our attention to the ideals that include  $x_1=x_{n+3}$ and these can be seen as ideals of regular fences. Let $\tilde{\alpha}_L = (\alpha_1, \ldots, \alpha_{2s}-1)$ and $\tilde{\alpha}_R = (\alpha_1-1, \ldots, \alpha_{2s})$ We claim that 
\[ \sum_{\substack{J\trianglelefteq \bar{\mathfrak{F}}(\alpha_L)\\x_1 = x_{n+3} \in J} }q^{|J|} =  q\,\sum_{J\trianglelefteq \mathfrak{F}(\tilde{\alpha}_L) }q^{|J|}, \qquad \sum_{\substack{J\trianglelefteq \bar{\mathfrak{F}}(\alpha_R)\\x_1 = x_{n+3} \in J} }q^{|J|} =  q\,\sum_{J\trianglelefteq \mathfrak{F}(\tilde{\alpha}_R) }q^{|J|}. \]

This is because the ideals of $\overline{F}(\alpha_L)$ that contain $x_1$ correspond exactly to ideals of $F((\alpha_1,\alpha_2,\ldots,\alpha_{2s}-1))$, only shifted by $q$. The other equality is similarly justified. Here, we slightly abuse notation to permit the first or last segments to possibly be $0$, something that claim $\mathbf{B(n)}$ permits us to do. Consequently
\begin{eqnarray*} 
\sum_{J\trianglelefteq \bar{\mathfrak{F}}(\alpha_L)}q^{|J|}  - \sum_{J\trianglelefteq \bar{\mathfrak{F}}(\alpha_R) }q^{|J|}  =& \sum_{\substack{J\trianglelefteq \bar{\mathfrak{F}}(\alpha_L)\\x_1 = x_{n+2} \in J} }q^{|J|} - \sum_{\substack{J\trianglelefteq \bar{\mathfrak{F}}(\alpha_R)\\x_1 = x_{n+2} \in J} }q^{|J|}  \\
=&  q\left(\sum_{J\trianglelefteq \mathfrak{F}(\tilde{\alpha}_L) }q^{|J|}, \qquad  -  \sum_{J\trianglelefteq \mathfrak{F}(\tilde{\alpha}_R) }q^{|J|}\right). 
\end{eqnarray*}

In the final expressions, we have compositions of $n$, so by Claim \ref{claim2} the difference between the generating polynomials of their ideals is symmetric around $n/2$. Shifting by $q$ gives a rank sequence symmetric around $(n+2)/2$ as desired.
\end{claimproof}

\begin{claim} \label{claim4} $\mathbf{D(n+2)}$ holds.

\end{claim}
\begin{claimproof} 
Let $\alpha = (\alpha_1, \ldots, \alpha_{2s})$ be a composition of $n+2$. Claim \ref{claim3} says that moving an element across a valley (as long as the number of parts does not change) preserves symmetry. Taking the vertical reflection of the poset, $\alpha_1$ yields another fence poset whose rank polynomial is the reflection of the original rank polyomial. 
\begin{eqnarray}\label{flip} 
\bar{R}(\alpha) = \sum_0^{n+2} r_k q^k \Longrightarrow \bar{R}(\alpha_1) = \sum_0^{n+2} r_{n-k} q^k.\end{eqnarray}
This is because for any $k$, there is a bijection between lower ideals of size $k$ of $\bar{F}(\alpha)$ and the lower ideals of size $n+2-k$ for $\bar{F}(\alpha_1)$, which is achieved by taking the set complement. 

Claim \ref{claim3} then shows that moving an element of $\alpha$ across a peak preserves symmetry as well. Applying these operations consecutively, we may transform $\alpha$ to a composition of the form $(k, 1, \ldots, 1)$ with $2s$ parts. By lemma \ref{lemma:basis}, this has symmetric rank polynomial. 
\end{claimproof}

As noted above, $\mathbf{D(2)}$ and $\mathbf{D(3)}$ are true by direct computations and the implications
\[\mathbf{D(n)} \implies \mathbf{A(n)} \implies \mathbf{B(n)} \implies \mathbf{C(n+1)} \implies \mathbf{D(n+2)},\]
 yield our theorem for all values of  $n$. 

\end{proof}

\begin{corollary} The polynomial $\overline{R}(\alpha;q)$ is invariant under cyclic shifts of segments of $\alpha$, so it is well defined over cyclic compositions.
\label{cor:cyclicshift}
\end{corollary}
\begin{proof}
if $\alpha = (\alpha_1, \ldots, \alpha_{2s})$ is a composition of $n$ and $\beta  = (\alpha_2, \ldots, \alpha_{2s}, \alpha_1)$, then we have that their rank polynomials are mirror images,
\[\bar{R}(\alpha) = \sum_0^n r_k q^k \Longrightarrow \bar{R}(\beta) = \sum_0^n r_{n-k} q^k,\]
as noted above in (\ref{flip}). Theorem \ref{thm:sym} yields the result. 

\end{proof}
\section{Proof of Main Theorem}

Given a rank sequence $(r_0,r_1,\ldots,r_{n+1})$, the properties of it being top interlacing, bottom interlacing or symmetric and unimodal  are determined by the relationship between elements whose indices are equidistant from $(n+1)/2$, which we will call $\text{mid}(\alpha)$. In all three cases, if $|j-\text{mid}(\alpha)|>|i-\text{mid}(\alpha)|$, then $r_j\leq r_i$, 

To this end, we will partition the inequalities that correspond to interlacing into two parts; the part that holds for both bottom and top interlacing sequences and the one that separates bottom and top interlacing sequences. 

\begin{eqnarray*}
\text{(ineqA)} & & r_0 \le r_{n},\, r_1 \le r_{n-1}\ldots \qquad \quad r_{n+1} \le r_{1},\,  r_{n} \le r_{2}\ldots\\
\text{(ineqB)}& & r_0\ge r_{n+1},\, r_1\ge r_{n},\, \ldots\\
\text{(ineqT)}& & r_0\le r_{n+1},\,r_1\le r_{n},\, \ldots
\end{eqnarray*}

Bottom interlacing sequences are ones that satisfy (ineqA) and (ineqB), top interlacing sequences are ones that satisfy (ineqA) and (ineqT), and symmetric unimodal ones are the ones that satisfy all three sets of inequalities.

\begin{proof}[Proof of Theorem \ref{thm:main}] Assume that the theorem holds for all compositions of length at most $n-1$. Let $\alpha$ be a composition of size $n$.
\begin{claim} The rank sequence $r(\alpha)=(r_0,r_1,\ldots,r_{n+1})$ satisfies (ineqA).
\end{claim}
\begin{claimproof} Following Methods $3$ and $4$ from Section \ref{sec:example}, we will consider two circular fences obtained by adding a new node to the fence of $\alpha$. We will  Let $\overline{F}(\alpha_T)$ be given by adding a node $x_{0}$ lying above both $x_1$ and $x_{n+1}$. Let $(t_0,t_1,\ldots,t_{n+2})$ be the corresponding rank sequence, which is symmetric by Theorem \ref{thm:sym}.

Note that the ideals of the fence poset of $\alpha$ correspond exactly to the ideals of $\overline{F}(\alpha_T)$ that do not contain $x_0$. The ideals that contain $x_0$ also contain $x_1, x_{n+1}$ and anything that is lying below them. 

\begin{eqnarray*}\displaystyle \overline{R}(\alpha_T;q)={R}(\alpha;q) +\sum_{\substack{I\trianglelefteq \overline{F}(\alpha_T)\\x_{0} \in I}}q^{|I|}\\
{R}(\alpha;q)=\overline{R}(\alpha_T;q)-q^k R(\beta;q).
\end{eqnarray*} where $\beta$ is obtained from $\alpha$ by deleting $x_1$, $x_{n+1}$ and anything below them, and $k$ is the number of nodes deleted $+1$.
Note that by the induction hypothesis, $R(\beta;q)$ is bottom or top interlacing, with $n-k+1$ nodes. For each symmetric pair $t_i$ and $t_{n+2-i}$, as ${n+2-i}$ is closer to the shifted center $k+(n-k+1)/2$, the amount subtracted from $t_{n+2-i}$ is at least as large as the amount subtracted from $t_i$, implying $r_i \geq r_{n+2-i}$ for $1\leq i \leq \lceil{n}\rceil$.

Similarly we can add a node $x_0$ that lies below both $x_1$ and $x_{n+1}$. Let  $\overline{F}(\alpha_B)$ be the corresponding circular fence poset with rank sequence $(b_0,b_1,\ldots,b_{n+2})$. By the same reasoning as above we get: 

\begin{eqnarray*}\displaystyle \overline{R}(\alpha_B;q)=q{R}(\alpha;q) +\sum_{\substack{I\trianglelefteq \overline{F}(\alpha_T)\\x_{0} \notin I}}q^{|I|}\\
q{R}(\alpha;q)=\overline{R}(\alpha_T;q)- R(\beta;q).
\end{eqnarray*}  where $\beta$ is obtained from $\alpha$ by deleting $x_1$, $x_{n+1}$ and anything above them. Now the center is shifted left, so that the amount subtracted from $b_{n+2-i}$ is less than or equal to the amount subtracted from $b_i$. As we shifted by $q$, this means $r_{i-1}\leq r_{n+1-i}$ for $1\leq i \leq \lceil{n}\rceil$. 
\end{claimproof}

To finish our proof, we will look at whether the rank sequence is bottom interlacing, top interlacing and symmetrical. 
To this end, we will add the new relation $x_1 \succeq x_n$, as in Method $2$ from \ref{sec:example}
\begin{claim} If $\alpha$ has an even number of segments, $R(\alpha;q)$ is bottom interlacing.
\end{claim}

\begin{claimproof} Let us add the relation $x_1\succeq x_{n+1}$ to $\alpha$. The resulting circular fence contains all ideals of ${F}(\alpha)$ satisfying $x_1 \in I \Rightarrow x_{n+1} \in I$. The ones that are left over are exactly the ones that contain $x_1$ but not $x_{n+1}$. The inclusion of $x_1$ is equivalent to deleting the node $x_1$ and shifting by $q$ and not including $x_{n+1}$ is equivalent to deleting the node $x_{n+1}$ as well as anything above it. What we are left with is the rank polynomial of a smaller composition $\beta$ shifted by $q$. As we have an even number of parts, there is at least one node above $x_{n+1}$ (See Figure \ref{fig:2113left} for an example) which means $\beta$ has at most $n-2$ nodes. So $\text{mid}(\beta)$, even when shifted by $q$ lies strictly to the left of $\frac{n+1}{2}$.

Let $(c_0,c_1,\ldots,c_{n+1})$ be the rank sequence of the circular fence, symmetric by Theorem \ref{thm:sym}. In particular, for each $i\leq \frac{n+1}{2}$, $c_i=c_{n+1-i}$. Adding the rank sequence for $\beta$ to this gives the rank sequence of $\alpha$. But $\text{mid}(\beta)$ laying strictly to the left of $\frac{n+1}{2}$ means for each $i$, what we add to $c_i$ is at least as large as what we add to $c_{n+1-i}$, giving us $r_i\geq r_{n+1-i}$.
\end{claimproof}

\begin{claim} If $\alpha$ has an odd number of segments, $R(\alpha;q)$ is bottom interlacing (respectively top interlacing) if and only if $R(\alpha';q)$ is bottom interlacing (respectively top interlacing) where $\alpha'=(\alpha_1-1,\alpha_2,\alpha_3,\ldots,\alpha_{s-1},\alpha_s -1)$ is the composition of $n-2$ obtained from $\alpha$ by subtracting $1$ from first and last segments (the fence of $\alpha'$ starts with a downwards segment if the first part is zero).

\end{claim}
\begin{claimproof} Again, we consider the circular fence obtained by adding the relation $x_1\succeq x_{n+1}$ to $\alpha$. The ideals of the circular fence are in bijection with ideals of ${F}(\alpha)$ satisfying $x_1 \in I \Rightarrow x_{n+1} \in I$. We will calculate the ones $x_1$ but not $x_{n+1}$ separately. As we have an odd number of parts, $x_1$ is below $x_2$ and $x_{n+1}$ is above $x_{n}$, so that the ideals containing $x_1$ but not $x_{n+1}$ are in bijection with the ideals of $\alpha'$ described above, with the rank sequence shifted by one. Deleting two nodes and shifting by one means that $\text{mid}(\alpha')=\text{mid}(\alpha)=\frac{n+1}{2}$. The rank polynomial of the circular fence is symmetric around $\frac{n+1}{2}$. Adding a bottom interlacing (respectively top interlacing) polynomial with the same $\text{mid}$ value gives us a bottom interlacing (resp. top interlacing) polynomial.
\end{claimproof}

Note that in the case of odd parts, if $\alpha_1>\alpha_s$ then removing pairs from both ends eventually gives us a fence with an even number of parts that is bottom interlacing, so $r(\alpha)$ is bottom interlacing. If $\alpha_1<\alpha_s$, looking at $\alpha^r=(\alpha_s,\alpha_{s-1},\ldots,\alpha_1)$ reverses the rank sequence, so $r(\alpha)$ is top interlacing. When $\alpha_1=\alpha_s$, removing pairs of nodes from both ends eventually gives us the fence for $(\alpha_2,\alpha_3,\ldots,\alpha_{s-1})$ turned upside down, whose rank sequence is the reverse of $r(\alpha_2,\alpha_3,\ldots,\alpha_{s-1})$. 
\end{proof}

\section{Rank Unimodality of Circular Fences}\label{sect:rankuni}

Unlike the regular case, the rank polynomial of circular fences is not always cyclic. In the case of $\alpha=( 1 , k , 1 , k )$, we get the rank sequence $[1, 2, 3, 4,\ldots, k,  k+1, k, k+1, k,\ldots, 3, 2, 1]$ which makes a slight dip in the middle (Refer to Figure \ref{fig:1417} for the rank lattice of $(1,5,1,5)$). We will next see that this issue can only happen when we have an even number of nodes, and a dip can only happen in the middle term of the rank sequence.

\begin{prop} \label{prop:unimod} If $\alpha=(\alpha_1,\alpha_2,\ldots,\alpha_{2s})$ has an odd number of nodes, then $\overline{R}(\alpha;q)$ is unimodal. If $\alpha$ is of size $2t$ for some $t \in \mathbb{N}$, then we have $r_i\geq r_{i-1}$ for all $i<t$.\end{prop}
\begin{proof} Take a composition $\alpha$ of $n$ and let $T$ be a maximal node in $\overline{F}(\alpha)$. We can partition ideals of $\alpha$ into two parts: those that contain $T$ (and necessarily anything below it), and those that do not contain $T$. As deleting $T$ does not place any restrictions on other nodes, the ones that do not contain $T$ correspond to a regular fence of a composition $\beta$ with $n-1$ nodes, unimodal with $\text{mid}{\beta}=\frac{n-1}{2}$ by Theorem \ref{thm:main}. The ones that contain $T$ also contain the $k\geq 2$ nodes that lie below $T$ in $\overline{F}(\alpha)$, and they are in bijection with the ideals of $F(\gamma)$ obtained from $\overline{F}(\alpha)$ by deleting those nodes. The rank polynomial for $\overline{F}(\alpha)$ satisfies:
\begin{eqnarray*}
\overline{R}(\alpha;q)&={R}(\beta_q)+q^{k+1}R(\gamma;q).
\end{eqnarray*}

Denote the rank sequences of $\alpha$ and $\beta$ by $(r_0,r_1,..,r_n)$ and $(b_0,b_1,\ldots,b_{n-1})$ respectively. We have $b_{n-i}\geq b_{n-i+1}$ for all $1\leq i\leq\frac{n-1}{2}$ by unimodality (we take $b_n=0$). As $R(\gamma;q)$ is also rank unimodal, and $q^k\text{mid}(\gamma)$ lies strictly to the right of $\text{mid}(\beta)=\frac{n-3}{2}$, the value we add to $b_{n-i}$ is at least as large as the value we add to $b_{n-i-1}$, giving us $r_{n-i}\geq r_{n-i+1}$ and by symmetry $r_i\geq r_{i-1}$ for all $i$ satisfying $1\leq i\leq\frac{n-3}{2}$.

If $n$ is odd, this means unimodality. If $n=2t$ is even, we do get any information about the ordering of $r_{t-1}$ and $r_{t}$, so it is possible to have a dip in the middle, which indeed happens for $\alpha \neq ( 1 , k , 1 , k )$ and $( k , 1 , k, 1)$ for $k \in \mathbb{N}$.
\end{proof}

\begin{conj} For any $\alpha \neq ( 1 , k , 1 , k )$ or $( k , 1 , k, 1)$ for some $k$, the rank sequence $\overline{R}(\alpha;q)$ is unimodal.
\end{conj}

If the segments were fully independent, we would naturally end up with a unimodal polynomial. The connections of maximal and minimal entries work to add some additional relations so that some configurations are not allowed, a relatively small number. What this conjecture is saying is, when we look at a larger number of parts, the configurations disallowed are not sufficient to offset the underlying unimodality. Though we were unable to prove this in all generality, the next result shows that if there are exceptions, they are indeed very rare.

\begin{lemma} Let $T$ be a maximal node in the cyclic fence $\overline{F}(\alpha)$, and let $F_{T^-}$ be the (possibly upside down) fence obtained by deleting $T$. If the rank polynomial $R_{T^-}(q)$ corresponding to $F_{T^-}$ is top interlacing, then $\overline{R}(\alpha;q)$ is rank unimodal.
\end{lemma}

\begin{proof} We have already shown unimodality when the number of nodes is odd, so let us focus on the case $\alpha$ is a composition of $2t$. Let $F_{T^+}$ denote the fence obtained by deleting $T$ and any node below $T$ with the corresponding rank polynomial $R_{T^+}(q)$ so that we have:
$$\overline{R}(\alpha;q)= R_{T^-}(q)+q^{k+1}R_{T^+}(q) $$ where $k$ is the number of nodes below $T$ in $\overline{F}(\alpha)$.

As $F_{T^-}$ is top interlacing, its rank sequence $(r_0,r_1,\ldots,r_{2s-1})$ satisfies $r{t-1}\leq t_m$. The rank sequence of $q^kR_{T^+}(q)$ is unimodal with the largest entry falling strictly to the right of position $t$, so that the number we add to $r_{t-1}$ to obtain the $t-1$st entry of the rank sequence of $\overline{F}(\alpha)$ is at least as large as the number we add to $r_{t}$. As we already showed the only issue might be in the middle in Proposition \ref{prop:unimod}, we are done.
\end{proof}

\begin{corollary} If $\alpha=(\alpha_1,\alpha_2,\ldots,\alpha_{2s}  )$ has two consecutive segments larger than one, or $3$ consecutive segments $k, 1, l$ with $|k-l|>1$, then $\overline{R}(\alpha;q)$ is unimodal.
\end{corollary} 

\begin{proof} If $\alpha$ has two consecutive segments larger than one, we can assume without loss of generality, by Corollary \ref{cor:cyclicshift}, that they meet at a top bead $T$. Deleting $T$ gives an upside down fence of an odd number of parts, so it is top interlacing. Similarly, in the case where we have consecutive segments $k, 1, l$ with $|k-l|>1$ by symmetry we can assume that $k$ is larger and $k$ and $1$ meet in a top node $T$. Deleting $T$ gives a fence with an even number of parts, first of which is $l$ and the last is $k-1$. As $k-1$ is strictly larger than $l$, the corresponding rank polynomial is top interlacing.
\end{proof}

The leftover cases can be fully analysed when we have a small number of parts. For example, if we have four parts, the only cases that are not covered are of forms $(1,k,1,k)$ and $(1,k,1,k+1)$. If we have $6$ parts, possible counter examples to unimodality must be of one of these forms: $(1,k,1,k,1,k)$,$(1,k,1,k,1,k+1)$,$(1,k,1,k+1,1,k+1)$,$(1,1,2,1,1,2)$.

\section{Rowmotion on Circular Fences}

We can identify the ideals of a fence with antichains on that fence, as any ideal is uniquely described by its maximal elements. Rowmotion acts on ideals by taking an ideal $I$ to the ideal $\rho(I)$  corresponding to the antichain given by the minimal elements of the complement of $I$. In their recent paper \cite{rowmotion}, Elizalde, Plante, Roby and Sagan explored rowmotion on fences, and gave homomesy and orbomesy results, many of which hold for the circular case as well.

In particular they gave a bijection between the orbits of rowmotion on $F(\alpha)$ and an object called an $\alpha$-tiling. Here, we introduce a natural analogue, the class of \emph{circular} $\alpha$-tilings:
\begin{defn} For a composition $\alpha=(\alpha_1,\alpha_2,\ldots,\alpha_{2s})$, a circular $\alpha$-tiling is a tiling of a rectangle $R_{2s}$ with $2s$ rows labeled $1,2,\ldots,2s$ from top to bottom and an infinite number of columns with yellow $1 \times 1$ tiles, red $2 \times 1$ tiles which are allowed to wrap around and black $1\times (\alpha_i-1)$ tiles in row $i$ satisfying the following properties:
\begin{enumerate}[label=\textbf{(\alph*)}]
    \item If there is at least one black tile in a row, then when the red tiles are ignored, the black and yellow tiles alternate in that row.
    \item If $i$ is odd, there is a red tile in a column covering rows $i$ and $i+1$ if and only if the next column contains two yellow tiles in those two rows.
    \item If $i$ is even, there is a red tile covering rows $i$ and $i+1$ if $i<2s$ and wrapping around to cover  $2s$ and $1$ if and only if the previous column contains two yellow tiles in those rows.
\end{enumerate} 
\end{defn}

We say that a red tile \emph{starts} at row $i$ if it covers $i, i+1$ or $i=2s$ and it covers $2s$ and $1$. Though it is by no means clear from the definition, the connection with rowmotion orbits which we will prove next in Lemma~\ref{lem:rowmotionbijection} implies that all such tilings are periodic. The period of an orbit $\OO$ will be called the \emph{size} of $\OO$, denoted $|\OO|$. We will visually represent tilings by drawing one such period and identify tilings that are cyclic shifts of each other horizontally.

Let the map $\overline{\phi}$ take an ideal of $\overline{F}(\alpha=(\alpha_1,\alpha_2,\ldots,\alpha_{2s}))$ to a $2s\times 1$ rectangle where box $i$ is colored yellow if the $i$th segment contains no maximal elements of $I$, red if it contains a shared maximal element and black if it contains an unshared maximal element. $\overline{\phi}$ can be seen as a map on taking orbits of rowmotion to infinite rectangles of $2s$ rows by seeing each iteration of the rowmotion operation as a new column (See Figure \ref{fig:rowmotion2113} for an example). The following result directly follows from the proof of the corresponding Lemma 2.2 in \cite{rowmotion} and contains no new ideas. The proof is therefore omitted. 

\begin{lemma}\label{lem:rowmotionbijection} The map $\overline{\phi}$ is a bijection between orbits of rowmotion on $\overline{F}(\alpha)$ and circular $\alpha$-tilings.
\end{lemma}

For the following discussion, we will identify each tiling with its corresponding orbit and use the two interchangeably. Note that the placement of red tiles uniquely determines an orbit as long as there is at least one red tile in a row, as yellow and black tiles alternate in the leftover spaces. When describing all orbits of a particular fence, we will often talk about the placement of the red tiles, leaving it up to the reader to verify that the construction indeed gives a valid orbit.

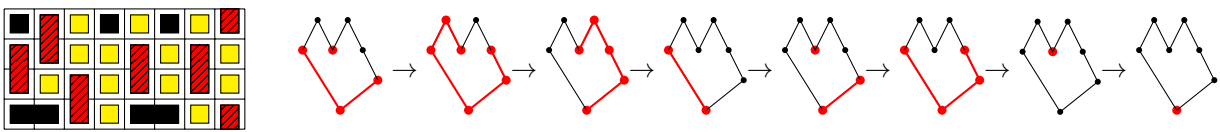
\begin{figure}[ht]\raisebox{-.2cm}{
\begin{tikzpicture}[scale=.4]
\draw(0,0)grid(8,-4);
\yrec{3}{2}\yrec{1}{3}\yrec{2}{3}\yrec{2}{4}\yrec{3}{4}\yrec{4}{4}
\yrec{1}{5}\yrec{2}{6}\yrec{3}{6}\yrec{1}{7}\yrec{4}{7}\yrec{2}{8}\yrec{3}{8}
\rrec{2}{1}\rrec{1}{2}\rrec{3}{3}\rrec{2}{5}\rrec{2}{7}\rrrec{4}{8}
\brec{1}{1}{1}\brec{1}{4}{1}\brec{1}{6}{1}\brec{4}{1}{2}\brec{4}{5}{2}
\end{tikzpicture}}\qquad
\begin{tikzpicture}[scale=.2]
\fill(3.5,-2) circle(.2);
\fill(1,2) circle(.2);
\fill(2,4) circle(.2);
\fill(3,2) circle(.2);
\fill(4,4) circle(.2);
\fill(5,2) circle(.2);
\fill(6,0) circle(.2);
\fill[red](3.5,-2) circle(.3);
\fill[red](1,2) circle(.3);
%\fill[red](2,4) circle(.3);
\fill[red](3,2) circle(.3);
%\fill[red](4,4) circle(.3);
%\fill[red](5,2) circle(.3);
\fill[red](6,0) circle(.3);
\draw (1,2)--(2,4)--(3,2)--(4,4)--(6,0);
\draw[red,thick] (1,2)--(3.5,-2)--(6,0);
\end{tikzpicture} \raisebox{.5cm}{$\rightarrow$}
\begin{tikzpicture}[scale=.2]
\fill(3.5,-2) circle(.2);
\fill(1,2) circle(.2);
\fill(2,4) circle(.2);
\fill(3,2) circle(.2);
\fill(4,4) circle(.2);
\fill(5,2) circle(.2);
\fill(6,0) circle(.2);
\fill[red](3.5,-2) circle(.3);
\fill[red](1,2) circle(.3);
\fill[red](2,4) circle(.3);
\fill[red](3,2) circle(.3);
%\fill[red](4,4) circle(.3);
\fill[red](5,2) circle(.3);
\fill[red](6,0) circle(.3);
\draw (1,2)--(2,4)--(3,2)--(4,4)--(6,0);
\draw[red,thick] (2,4)--(1,2)--(3.5,-2)--(6,0)--(5,2) (3,2)--(2,4);
\end{tikzpicture}\raisebox{.5cm}{$\rightarrow$}
\begin{tikzpicture}[scale=.2]
\fill(3.5,-2) circle(.2);
\fill(1,2) circle(.2);
\fill(2,4) circle(.2);
\fill(3,2) circle(.2);
\fill(4,4) circle(.2);
\fill(5,2) circle(.2);
\fill(6,0) circle(.2);
\fill[red](3.5,-2) circle(.3);
%\fill[red](1,2) circle(.3);
%\fill[red](2,4) circle(.3);
\fill[red](3,2) circle(.3);
\fill[red](4,4) circle(.3);
\fill[red](5,2) circle(.3);
\fill[red](6,0) circle(.3);
\draw (3.5,-2)--(1,2)--(2,4)--(3,2)--(4,4)--(6,0);
\draw[red,thick] (3.5,-2)--(6,0)--(4,4)--(3,2);
\end{tikzpicture}\raisebox{.5cm}{$\rightarrow$}
\begin{tikzpicture}[scale=.2]
\fill(3.5,-2) circle(.2);
\fill(1,2) circle(.2);
\fill(2,4) circle(.2);
\fill(3,2) circle(.2);
\fill(4,4) circle(.2);
\fill(5,2) circle(.2);
\fill(6,0) circle(.2);
\fill[red](3.5,-2) circle(.3);
\fill[red](1,2) circle(.3);
%\fill[red](2,4) circle(.3);
%\fill[red](3,2) circle(.3);
%\fill[red](4,4) circle(.3);
%\fill[red](5,2) circle(.3);
%\fill[red](6,0) circle(.3);
\draw (3.5,-2)--(1,2)--(2,4)--(3,2)--(4,4)--(6,0)--(3.5,-2);
\draw[red,thick] (1,2)--(3.5,-2);
\end{tikzpicture}\raisebox{.5cm}{$\rightarrow$}
\begin{tikzpicture}[scale=.2]
\fill(3.5,-2) circle(.2);
\fill(1,2) circle(.2);
\fill(2,4) circle(.2);
\fill(3,2) circle(.2);
\fill(4,4) circle(.2);
\fill(5,2) circle(.2);
\fill(6,0) circle(.2);
\fill[red](3.5,-2) circle(.3);
%\fill[red](1,2) circle(.3);
%\fill[red](2,4) circle(.3);
\fill[red](3,2) circle(.3);
%\fill[red](4,4) circle(.3);
%\fill[red](5,2) circle(.3);
\fill[red](6,0) circle(.3);
\draw (3.5,-2)--(1,2)--(2,4)--(3,2)--(4,4)--(6,0);
\draw[red,thick] (3.5,-2)--(6,0);
\end{tikzpicture}\raisebox{.5cm}{$\rightarrow$}
\begin{tikzpicture}[scale=.2]
\fill(3.5,-2) circle(.2);
\fill(1,2) circle(.2);
\fill(2,4) circle(.2);
\fill(3,2) circle(.2);
\fill(4,4) circle(.2);
\fill(5,2) circle(.2);
\fill(6,0) circle(.2);
\fill[red](3.5,-2) circle(.3);
\fill[red](1,2) circle(.3);
%\fill[red](2,4) circle(.3);
%\fill[red](3,2) circle(.3);
%\fill[red](4,4) circle(.3);
\fill[red](5,2) circle(.3);
\fill[red](6,0) circle(.3);
\draw (3.5,-2)--(1,2)--(2,4)--(3,2)--(4,4)--(6,0);
\draw[red,thick] (1,2)--(3.5,-2)--(6,0)--(5,2);
\end{tikzpicture}\raisebox{.5cm}{$\rightarrow$}
\begin{tikzpicture}[scale=.2]
\fill(3.5,-2) circle(.2);
\fill(1,2) circle(.2);
\fill(2,4) circle(.2);
\fill(3,2) circle(.2);
\fill(4,4) circle(.2);
\fill(5,2) circle(.2);
\fill(6,0) circle(.2);
\fill[red](3,2) circle(.3);
\draw (3.5,-2)--(1,2)--(2,4)--(3,2)--(4,4)--(6,0)--(3.5,-2);
\end{tikzpicture}\raisebox{.5cm}{$\rightarrow$}
\begin{tikzpicture}[scale=.2]
\fill(3.5,-2) circle(.2);
\fill(1,2) circle(.2);
\fill(2,4) circle(.2);
\fill(3,2) circle(.2);
\fill(4,4) circle(.2);
\fill(5,2) circle(.2);
\fill(6,0) circle(.2);
\fill[red](3.5,-2) circle(.3);
\draw (3.5,-2)--(1,2)--(2,4)--(3,2)--(4,4)--(6,0)--(3.5,-2);
\end{tikzpicture}
    \caption{A circular $(2,1,1,3)$-tiling and the corresponding orbit of rowmotion on $\overline{F}(2,1,1,3)$}\label{fig:rowmotion2113}
\end{figure}

A statistic $\text{st}$ is said to be $d$-mesic (with respect to a group operation) if its average is $d$ on every orbit, and it is said to be homomesic if it is $d$-mesic for some $d$. We will consider the following statistics on orbits of rowmotion on cyclic fences:

\begin{eqnarray*}
\MM_x(\OO)=\text{number of times $x$ occurs as a maximal element in } \OO,\qquad &&\MM(\OO)=\sum_x\MM_x(\OO),\\
\overline{\chi}_x(\OO)=\text{number of times $x$ occurs in } \OO,\qquad &&\overline{\chi}(\OO)=\sum_x\overline{\chi}_x(\OO).
\end{eqnarray*}

We can read the values of the statistics described directly from the tiling. Let $b_i$ and $w_i$ denote the number of black tiles and white on row $i$ on one period of $\OO$ respectively, and $r_i$ denote the number of red tiles starting in row $i$. Note that as black tiles only occur alternating with white tiles, $b_i=w_i$ in any row where $b_i$ is non-zero. So, for any row $i$, the total $w_i(\alpha_i)+r_i+r_{i-1}$ is equal to the period of $\OO$. An unshared element $x$ on segment $i$ occurs as a maximal element once per every black tile on row $i$, and the occurrences of shared elements correspond to red tiles. 

If $\alpha_i\geq 2$ and $x$ is the $j$th smallest unshared element on segment $i$, we get:
\begin{eqnarray*}
\MM_x(\OO)= b_i, \qquad  \overline{\chi}_x(\OO)=\begin{cases}
b_i(\alpha_i-j)+r_i& \text{if $i$ is odd}\\
b_i(\alpha_i-j)+r_{i-1}& \text{if $i$ is even}.
\end{cases}
\end{eqnarray*}

Similarly for a maximal element $T$ lying between segments $2i+1$ and $2i+2$, and a minimal element $B$ lying between segments $2i$ and $2i+1$ (cyclically), we have:

\begin{eqnarray*}
\overline{\chi}_T(\OO)&=& r_{2i+1}=\MM_T(\OO), \quad \quad
\MM_B(\OO)=r_{2i},  \\ \overline{\chi}_B(\OO)&=&|\OO|-r_{2i}=r_{2i-1}+w_{2i}(\alpha_{2i})=r_{2i+1}+w_{2i+1}(\alpha_{2i+1}).
\end{eqnarray*}

By summing up these values over all nodes of the fence, we get the following formulas (using the convention $\alpha_{2s+1}=\alpha_1$):
\begin{eqnarray}
\MM(\OO)&=&\sum_{i\leq2s}b_i(\alpha_i-1)+r_i,\label{eq:maxorbit}\\
\overline{\chi}(\OO)&=&s|\OO|+\sum_{i\leq2s}b_i\binom{\alpha_i}{2}+\sum_{i\leq s}r_{2i-1}(\alpha_{2i-1}+\alpha_{2i}-1)-r_{2i}.\\
&=&m/2\,|\OO|-\sum_{i\leq2s}(-1)^i r_i(\alpha_i+\alpha_{i+1})/2\label{eq:sumorbit}
\end{eqnarray}

We have shown in Theorem \ref{thm:sym} that the rank polynomial for circular fences is always symmetric.  This means that if the statistic $\overline{\chi}$ is homomesic, it is necessarily $m\slash 2$-mesic. So in a way the last part of Equation \ref{eq:sumorbit} describes how far from an homomesy an orbit is. If   $\alpha_i+\alpha_{i+1}$ is the same for all $i$, as in the example of $(3,1,3,1)$ below, then $\overline{\chi}$ is $m\slash 2$-mesic if and only if $\sum_{i\leq s} r_{2i}-r_{2i-1}=0$ for all orbits.

\begin{example}[$\overline{F}(3,1,3,1)$.]\label{ex:3131}
On the small case $(3,1,3,1)$, row motion has $3$ orbits, one of size $5$ and two of size $9$.

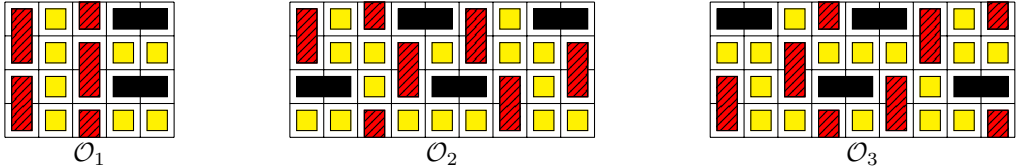
\vspace{2mm}

\begin{tikzpicture}[scale=.45]
\draw(0,0)grid(5,-4);
\rrec{1}{1} \rrec{3}{1} \rrec{2}{3} \rrrec{4}{3} \yrec{1}{2} \yrec{2}{2} \yrec{3}{2} \yrec{4}{2} 
\brec{1}{4}{2} \brec{3}{4}{2} \yrec{2}{4} \yrec{2}{5} \yrec{4}{4} \yrec{4}{5}
\draw node at (2.5,-4.5) {$\OO_1$};
\end{tikzpicture} \qquad \qquad
\begin{tikzpicture}[scale=.45]
\draw(0,0)grid(9,-4);
\rrec{1}{1}\rrrec{4}{3}\rrec{2}{4}\rrec{1}{6}\rrec{3}{7}
\yrec{4}{1}\yrec{4}{2}\yrec{4}{5}\yrec{4}{6}\yrec{4}{4} \yrec{4}{8}
\yrec{2}{2}\yrec{2}{3}\yrec{2}{5}\yrec{2}{7}\yrec{2}{8}
\yrec{1}{2}\yrec{1}{7}\yrec{3}{3}\yrec{3}{8}\brec{1}{4}{2}\brec{1}{8}{2}\brec{3}{1}{2} 
\brec{3}{5}{2} \yrec{2}{9} \yrec{4}{9} \rrec{2}{9}
\draw node at (4.5,-4.5) {$\OO_2$};
\end{tikzpicture}  
\qquad \qquad
\begin{tikzpicture}[scale=.45]
\draw(0,0)grid(9,-4);
\rrec{3}{1}\rrec{2}{3}\rrrec{4}{4}\rrec{3}{6}\rrec{1}{7}
\yrec{2}{1}\yrec{2}{2}\yrec{2}{5}\yrec{2}{6}\yrec{2}{4} \yrec{2}{8}
\yrec{4}{2}\yrec{4}{3}\yrec{4}{5}\yrec{4}{7}\yrec{4}{8}
\yrec{3}{2}\yrec{3}{7}\yrec{1}{3}\yrec{1}{8}\brec{3}{4}{2}\brec{3}{8}{2}\brec{1}{1}{2} \rrrec{4}{9}
\brec{1}{5}{2} \yrec{2}{9} 
 \draw node at (4.5,-4.5) {$\OO_3$};
\end{tikzpicture}   \centering

\begin{flushleft}
 We can calculate the values of $\overline{\chi}$ and $\MM$ statistics via Equations $\ref{eq:maxorbit}$-$\ref{eq:sumorbit}$: \end{flushleft}
\begin{eqnarray*}
\MM(\OO)=2(b_1+b_3)+(r_1+r_2+r_3+r_4),& \quad \quad &\overline{\chi}(\OO)=4|\OO|+4(r_1-r_2+r_3-r_4),\\
\MM(\OO_1)=2(2)+4=8,& \quad \quad & \MM(\OO_2)=\MM(\OO_3)=2(4)+6=14,\\
\overline{\chi}(\OO_1)=4(5)=24,& \quad \quad & \overline{\chi}(\OO_2)=\overline{\chi}(\OO_3)=4(9)=36.
\end{eqnarray*}
\begin{flushleft}
  Note that the second $9$-orbit can be obtained from the first by shifting rows cyclically by $2$ so it makes sense that they have the same statistics. The statistic $\overline{\chi}$ is $4$-mesic.
\end{flushleft}
\end{example}

Applying the formulas for the $\MM$ and $\overline{\chi}$ statistics, we see that many homomesy results from the non-circular fences also apply for the circular ones:

\begin{prop}\label{prop:homomesy} For a composition $\alpha=(\alpha_1,\alpha_2,\ldots,\alpha_{2s})$, rowmotion operation on the circular fence $\overline{F}(\alpha)$ has the following properties:
\begin{enumerate}
    \item If $x$ and $y$ are unshared elements on the same segment, $\MM_x-\MM_y$ is $0$-mesic.
    \item For an unshared element $x$ of segment $i$ that lies between a maximal element $T$ and a minimal element $B$,  $\MM_x \alpha_i+\MM_{T}+\MM_{B}$ is $1$-mesic.
    \item For a maximal element $T$ lying between segments $2i+1$ and $2i+2$, and a minimal element $B$ lying between segments $2j$ and $2j+1$ (cyclically), if $r_{2i+1}=r_{2j}$ for all orbits $\OO$, then $\overline{\chi}_T+\overline{\chi}_B$ is $1$-mesic.
    \item If $\alpha_i=2$ for all $i$, then $\MM$ is $s$-mesic.
\end{enumerate}
\end{prop}

We have previously noted that taking setwise complements maps the ideals of $\overline{F}(\alpha)$ to ideals of $\overline{\text{sh}(\alpha)}$,the fence of a cyclic shift of $\alpha$ by one step. We can also see $\kappa$ as the map taking a circular $\alpha$-tiling, doing a vertical cyclic shift of one step and a horizontal flip to get a circular $\text{sh}(\alpha)$-tiling. Figure \ref{fig:shrowmotion2113} shows the action of $\kappa$ on the orbit seen in Figure \ref{fig:rowmotion2113}.
As the rowmotion is defined via the complement operation, it is quite well behaved under this map.

\begin{lemma} \label{lem:rowmotionflip} Let $\kappa$ denote the complement map between ideals of $\overline{F}(\alpha)$ and $\overline{F}(\text{sh}(\alpha))$. Then for any ideal $I$ we have $\kappa(\partial(I))=\partial^{-1}(\kappa(I))$, meaning $\kappa$ maps orbits to orbits. In particular, if $|\alpha|=m-n$, for any orbit $\OO$ of rowmotion on  $\overline{F}(\alpha)$ we have:
\begin{eqnarray*}
\MM(\OO)&=&\MM(\kappa(\OO))\\
 \overline{\chi}(\OO)+ \overline{\chi}(\kappa(\OO))&=&{n}|\OO|.
\end{eqnarray*}
\end{lemma}
\begin{proof} As for any ideal $I$ $\overline{\chi}(I)+\overline{\chi}(\kappa(I))=m$, the second statement is trivial. The first is slightly more complicated as we do not necessarily have $\MM(I)=\MM(\kappa(I))$, for example if $I$ is the empty ideal, $\MM(I)=0$ whereas $\MM(\kappa(I))=s$, where $2s$ is the length of $\alpha$. However, as the total number of red and black tiles remains unchanged under $\kappa$, the result follows by Equation \ref{eq:maxorbit}.
\end{proof}

\begin{figure}[ht]\raisebox{-.2cm}{
\begin{tikzpicture}[scale=.4]
\draw(0,0)grid(8,-4);
\yrec{4}{7}\yrec{2}{6}\yrec{3}{6}\yrec{3}{5}\yrec{4}{5}\yrec{1}{5}
\yrec{2}{4}\yrec{3}{3}\yrec{4}{3}\yrec{2}{2}\yrec{1}{2}\yrec{3}{1}\yrec{4}{1}
\rrec{3}{8}\rrec{2}{7}\rrrec{4}{6}\rrec{3}{4}\rrec{3}{2}\rrec{1}{1}
\brec{2}{8}{1}\brec{2}{5}{1}\brec{2}{3}{1}\brec{1}{7}{2}\brec{1}{3}{2}
\end{tikzpicture}}\qquad
\begin{tikzpicture}[scale=.2,rotate=180]
\fill(-3.5,-2) circle(.2);
\fill(-1,2) circle(.2);
\fill(-2,4) circle(.2);
\fill(-3,2) circle(.2);
\fill(-4,4) circle(.2);
\fill(-5,2) circle(.2);
\fill(-6,0) circle(.2);
%\fill[red](-3.5,-2) circle(.3);
%\fill[red](-1,2) circle(.3);
\fill[red](-2,4) circle(.3);
%\fill[red](-3,2) circle(.3);
\fill[red](-4,4) circle(.3);
\fill[red](-5,2) circle(.3);
%\fill[red](-6,0) circle(.3);
\draw (-3.5,-2)--(-1,2)--(-2,4)--(-3,2)--(-4,4)--(-6,0)--(-3.5,-2);
\draw[red,thick] (-4,4)--(-5,2);
\end{tikzpicture} \raisebox{.5cm}{$\rightarrow$}
\begin{tikzpicture}[scale=.2,rotate=180]
\fill(-3.5,-2) circle(.2);
\fill(-1,2) circle(.2);
\fill(-2,4) circle(.2);
\fill(-3,2) circle(.2);
\fill(-4,4) circle(.2);
\fill(-5,2) circle(.2);
\fill(-6,0) circle(.2);
%\fill[red](-3.5,-2) circle(.3);
%\fill[red](-1,2) circle(.3);
%\fill[red](-2,4) circle(.3);
%\fill[red](-3,2) circle(.3);
\fill[red](-4,4) circle(.3);
%\fill[red](-5,2) circle(.3);
%\fill[red](-6,0) circle(.3);
\draw (-3.5,-2)--(-1,2)--(-2,4)--(-3,2)--(-4,4)--(-6,0)--(-3.5,-2);
%\draw[red,thick] (-2,4)--(-1,2)--(-3.5,-2)--(-6,0)--(-5,2) (-3,2)--(-2,4);
\end{tikzpicture}\raisebox{.5cm}{$\rightarrow$}
\begin{tikzpicture}[scale=.2,rotate=180]
\fill(-3.5,-2) circle(.2);
\fill(-1,2) circle(.2);
\fill(-2,4) circle(.2);
\fill(-3,2) circle(.2);
\fill(-4,4) circle(.2);
\fill(-5,2) circle(.2);
\fill(-6,0) circle(.2);
%\fill[red](3.5,-2) circle(.3);
\fill[red](-1,2) circle(.3);
\fill[red](-2,4) circle(.3);
%\fill[red](3,2) circle(.3);
%\fill[red](4,4) circle(.3);
%\fill[red](5,2) circle(.3);
%\fill[red](6,0) circle(.3);
\draw (-3.5,-2)--(-1,2)--(-2,4)--(-3,2)--(-4,4)--(-6,0)--(-3.5,-2);
\draw[red,thick] (-2,4)--(-1,2);
\end{tikzpicture}\raisebox{.5cm}{$\rightarrow$}
\begin{tikzpicture}[scale=.2,rotate=180]
\fill(-3.5,-2) circle(.2);
\fill(-1,2) circle(.2);
\fill(-2,4) circle(.2);
\fill(-3,2) circle(.2);
\fill(-4,4) circle(.2);
\fill(-5,2) circle(.2);
\fill(-6,0) circle(.2);
%\fill[red](3.5,-2) circle(.3);
%\fill[red](1,2) circle(.3);
\fill[red](-2,4) circle(.3);
\fill[red](-3,2) circle(.3);
\fill[red](-4,4) circle(.3);
\fill[red](-5,2) circle(.3);
\fill[red](-6,0) circle(.3);
\draw (-3.5,-2)--(-1,2)--(-2,4)--(-3,2)--(-4,4)--(-6,0)--(-3.5,-2);
\draw[red,thick] (-2,4)--(-3,2)--(-4,4)--(-6,0);
\end{tikzpicture}\raisebox{.5cm}{$\rightarrow$}
\begin{tikzpicture}[scale=.2,rotate=180]
\fill(-3.5,-2) circle(.2);
\fill(-1,2) circle(.2);
\fill(-2,4) circle(.2);
\fill(-3,2) circle(.2);
\fill(-4,4) circle(.2);
\fill(-5,2) circle(.2);
\fill(-6,0) circle(.2);
%\fill[red](3.5,-2) circle(.3);
\fill[red](-1,2) circle(.3);
\fill[red](-2,4) circle(.3);
%\fill[red](3,2) circle(.3);
\fill[red](-4,4) circle(.3);
\fill[red](-5,2) circle(.3);
%\fill[red](6,0) circle(.3);
\draw (-3.5,-2)--(-1,2)--(-2,4)--(-3,2)--(-4,4)--(-6,0)--(-3.5,-2);
\draw[red,thick](-1,2)--(-2,4) (-4,4)--(-5,2);
\end{tikzpicture}\raisebox{.5cm}{$\rightarrow$}
\begin{tikzpicture}[scale=.2,rotate=180]
\fill(-3.5,-2) circle(.2);
\fill(-1,2) circle(.2);
\fill(-2,4) circle(.2);
\fill(-3,2) circle(.2);
\fill(-4,4) circle(.2);
\fill(-5,2) circle(.2);
\fill(-6,0) circle(.2);
%\fill[red](3.5,-2) circle(.3);
%\fill[red](1,2) circle(.3);
\fill[red](-2,4) circle(.3);
\fill[red](-3,2) circle(.3);
\fill[red](-4,4) circle(.3);
%\fill[red](5,2) circle(.3);
%\fill[red](6,0) circle(.3);
\draw(-3.5,-2)--(-1,2)--(-2,4)--(-3,2)--(-4,4)--(-6,0)--(-3.5,-2);
\draw[red,thick] (-2,4)--(-3,2)--(-4,4);
\end{tikzpicture}\raisebox{.5cm}{$\rightarrow$}
\begin{tikzpicture}[scale=.2,rotate=180]
\fill(-3.5,-2) circle(.2);
\fill(-1,2) circle(.2);
\fill(-2,4) circle(.2);
\fill(-3,2) circle(.2);
\fill(-4,4) circle(.2);
\fill(-5,2) circle(.2);
\fill(-6,0) circle(.2);
\fill[red](-3.5,-2) circle(.3);
\fill[red](-1,2) circle(.3);
\fill[red](-2,4) circle(.3);
%\fill[red](-3,2) circle(.3);
\fill[red](-4,4) circle(.3);
\fill[red](-5,2) circle(.3);
\fill[red](-6,0) circle(.3);
\draw (-3.5,-2)--(-1,2)--(-2,4)--(-3,2)--(-4,4)--(-6,0)--(-3.5,-2);
\draw[red,thick] (-3.5,-2)--(-1,2)--(-2,4) (-4,4)--(-6,0)--(-3.5,-2);
\end{tikzpicture}\raisebox{.5cm}{$\rightarrow$}
\begin{tikzpicture}[scale=.2,rotate=180]
\fill(-3.5,-2) circle(.2);
\fill(-1,2) circle(.2);
\fill(-2,4) circle(.2);
\fill(-3,2) circle(.2);
\fill(-4,4) circle(.2);
\fill(-5,2) circle(.2);
\fill(-6,0) circle(.2);
%\fill[red](3.5,-2) circle(.3);
\fill[red](-1,2) circle(.3);
\fill[red](-2,4) circle(.3);
\fill[red](-3,2) circle(.3);
\fill[red](-4,4) circle(.3);
\fill[red](-5,2) circle(.3);
\fill[red](-6,0) circle(.3);
\draw (-3.5,-2)--(-1,2)--(-2,4)--(-3,2)--(-4,4)--(-6,0)--(-3.5,-2);
\draw[red,thick] (-1,2)--(-2,4)--(-3,2)--(-4,4)--(-6,0);
\end{tikzpicture}
    \caption{A circular $(3,2,1,1)$-tiling and the corresponding orbit}\label{fig:shrowmotion2113}
\end{figure}
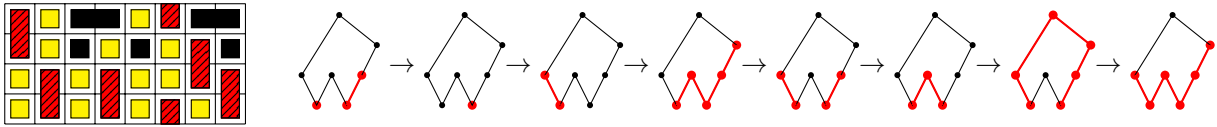

When we have only two parts of sizes $a$ and $b$, the rank lattice becomes a $a\times b$ lattice with added minimum and maximum elements (see Figure \ref{fig:48latticechains} for an example). The bijection with tilings allo.ws us to easily describe the orbits in terms of the lcm and gcd of the two segments.

\begin{prop} Let $d=\text{gcd}(a,b)$ and $m=\text{lcm}(a,b)$. Then, rowmotion on $\overline{F}(a,b)$ has a unique orbit of size $m+2$ and $d-1$ orbits of size $m$, where $\MM$ takes the values $2(a+b)(m+2)/d$ and $2m-(a+b)/d$ respectively. The statistic $\overline{\chi}$ is $(a+b)/2$-mesic.
\end{prop}

\begin{proof} There is a unique tiling of size $m+2$ that contains red tiles, one starting on the first row of column $1$, and the other starting at the third row of column $3$. The rest of the orbits contain no red tiles, and are given by the $d-1$ ways of placing the white tiles so that they always fall on different columns (see Figure \ref{fig:48latticeorbits} for an example). In all orbits $w_1=b/d$ and $w_2=a/d$, and plugging in these values to equations \ref{eq:maxorbit} and \ref{eq:sumorbit} allows us the calculate $\MM$ and $\overline{\chi}$.
\end{proof}

Another way to visualise the orbits in this case is to think of them as walks on the rank lattice using the moves $(1,1)$,$(a,-1)$,$(-1,b)$ and the special move that connects the maximum and the minimum, refer to Figure \ref{fig:48latticeorbits} for the example of $\overline{F}(4,8)$.
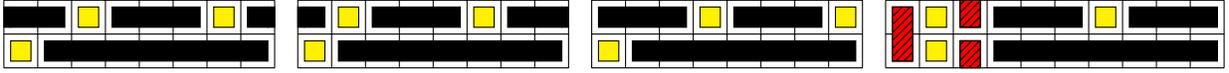
\begin{figure}[ht]
    \centering
    \include{latticeorbits}
    
    \vspace{.1cm}
    
     \begin{tikzpicture}[scale=.45]
\draw(0,0)grid(8,-2);
\brec{1}{1}{2} \brec{1}{4}{3} \brec{1}{8}{1} \brec{2}{2}{7}
\yrec{1}{3}\yrec{1}{7}\yrec{2}{1}
\filldraw (0,-0.2) rectangle (.2,-.8); \filldraw (8,-0.2) rectangle (7.8,-.8);
\end{tikzpicture} $\,$
 \begin{tikzpicture}[scale=.45]
\draw(0,0)grid(8,-2);
\brec{1}{1}{1} \brec{1}{3}{3} \brec{1}{7}{2} \brec{2}{2}{7} \filldraw (0,-0.2) rectangle (.2,-.8); \filldraw (8,-0.2) rectangle (7.8,-.8);
\yrec{1}{2}\yrec{1}{6}\yrec{2}{1}
\end{tikzpicture} $\,$
 \begin{tikzpicture}[scale=.45]
\draw(0,0)grid(8,-2);
\brec{1}{1}{3} \brec{1}{5}{3} \brec{2}{2}{7} 
\yrec{1}{4}\yrec{1}{8}\yrec{2}{1}
\end{tikzpicture} $\,$
\begin{tikzpicture}[scale=.45]
\draw(0,0)grid(10,-2);
\brec{1}{4}{3} \brec{1}{8}{3} \brec{2}{4}{7} 
\yrec{1}{2}\yrec{1}{7}\yrec{2}{2}
\rrec{1}{1} \rrrec{2}{3} 
\end{tikzpicture} 
    \caption{The four orbits of rowmotion on $\overline{F}(4,8)$ as paths on $\overline{J}(4,8)$ and corresponding tilings.}
    \label{fig:48latticeorbits}
\end{figure}

\subsection{Other Cases With Few Segments}

When the number of segments is small, it is often possible to build all orbits from ones of smaller size, by \emph{dilating} orbits for partitions of small size, where we add new columns that lengthen the black tiles without creating problems. Figure \ref{fig:4dilation} shows an example where adding new columns to the marked spaces is how we build the orbits for larger partitions. In this section, we will use "dilation" arguments to fully describe the action of rowmotion on $\overline{F}((1,1,a,1))$ and $\overline{F}((a,1,a,1))$. The idea can be extended to use the orbits for $(k,1,a,1)$ to build the orbits of $(k,1,a+k+3,1)$ for general $k$.

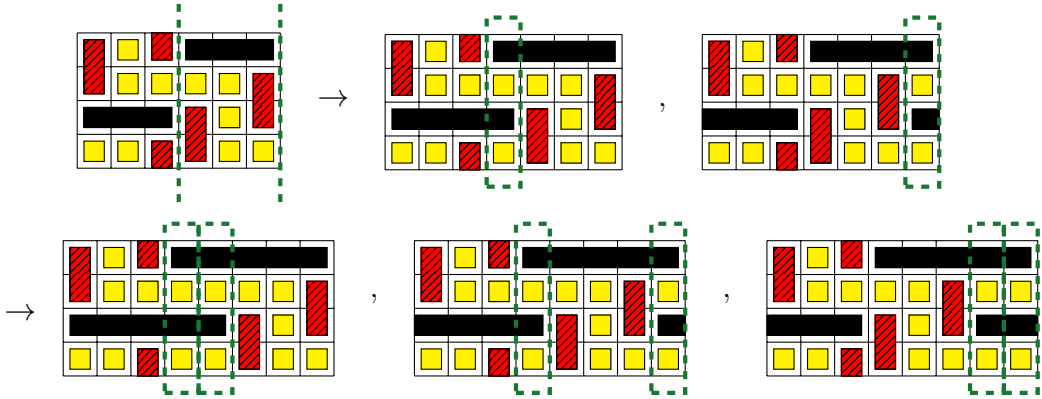
\begin{figure}[ht]
    \centering
\begin{tikzpicture}[scale=.45]
\draw(0,0)grid(6,-4);
\rrec{1}{1} \rrec{2}{6} \rrec{3}{4} \rrrec{4}{3}
\yrec{1}{2} \yrec{2}{2} \yrec{3}{5} \yrec{4}{5} \yrec{4}{2}
\yrec{2}{5} \yrec{3}{5},\brec{1}{4}{3} \brec{3}{1}{3} \yrec{2}{3} \yrec{2}{4}
\yrec{4}{1} \yrec{4}{6}
\draw[ultra thick,ultragreen, dashed] (3,-5)--(3,1) (6,-5)--(6,1);
\end{tikzpicture} \raisebox{1.3cm}{$\quad \mathlarger{\mathlarger{\rightarrow}} \quad$}
\begin{tikzpicture}[scale=.45]
\draw(0,0)grid(7,-4);
\rrec{1}{1} \rrec{2}{7} \rrec{3}{5} \rrrec{4}{3}
\yrec{1}{2} \yrec{2}{2} \yrec{3}{6} \yrec{4}{6} \yrec{4}{2} \yrec{4}{7} \yrec{2}{5} 
\yrec{2}{6} \yrec{3}{6},\brec{1}{4}{4} \brec{3}{1}{4} \yrec{2}{3} \yrec{2}{3} \yrec{4}{4}\yrec{2}{4}
\yrec{4}{1} \yrec{4}{6} \draw[white] (1,-5)--(2,-5) (3,1)--(6,1);
\draw[ultra thick,ultragreen, dashed] (3,-4.5)rectangle(4,.5);
\end{tikzpicture} \raisebox{1.3cm}{$\quad \mathlarger{\mathlarger{,}}  \quad$}
\begin{tikzpicture}[scale=.45]
\draw(0,0)grid(7,-4);
\rrec{1}{1} \rrec{2}{6} \rrec{3}{4} \rrrec{4}{3}
\yrec{1}{2} \yrec{2}{2} \yrec{3}{5} \yrec{4}{5} \yrec{4}{2}
\yrec{2}{5} \yrec{3}{5} \brec{1}{4}{4} \brec{3}{1}{3} \yrec{2}{3} \yrec{2}{4}
\yrec{4}{1} \yrec{4}{6} \brec{3}{7}{1} \yrec{2}{7} \yrec{4}{7}
\filldraw (0,-2.2) rectangle (.2,-2.8);
\filldraw (6.8,-2.2) rectangle (7,-2.8); \draw[white] (1,-5)--(2,-5) (3,1)--(6,1);
\draw[ultra thick,ultragreen, dashed] (6,-4.5)rectangle(7,.5);
\end{tikzpicture} 

\vspace{2mm}

 \raisebox{1cm}{$\quad \mathlarger{\mathlarger{\rightarrow}}  \quad$}\begin{tikzpicture}[scale=.45]
\draw(0,0)grid(8,-4);
\rrec{1}{1} \rrec{2}{8} \rrec{3}{6} \rrrec{4}{3}
\yrec{1}{2} \yrec{2}{2} \yrec{3}{7} \yrec{4}{7} \yrec{4}{2} \yrec{4}{8} \yrec{2}{6} 
\yrec{2}{7} \yrec{3}{7},\brec{1}{4}{5} \brec{3}{1}{5} \yrec{2}{3} \yrec{2}{3} \yrec{4}{5}\yrec{2}{5}
\yrec{4}{1} \yrec{4}{7} \yrec{2}{4} \yrec{4}{4} \draw[ultra thick,ultragreen, dashed] (3,-4.5)rectangle(4,.5);
\draw[ultra thick,ultragreen, dashed] (4,-4.5)rectangle(5,.5);
\end{tikzpicture} \raisebox{1.3cm}{$\quad \mathlarger{\mathlarger{,}}  \quad$}
\begin{tikzpicture}[scale=.45]
\draw(0,0)grid(8,-4);
\rrec{1}{1} \rrec{2}{7} \rrec{3}{5} \rrrec{4}{3}
\yrec{1}{2} \yrec{2}{2} \yrec{3}{6} \yrec{4}{6} \yrec{4}{2} \yrec{4}{7} \yrec{2}{5} 
\yrec{2}{6} \yrec{3}{6},\brec{1}{4}{5} \brec{3}{1}{4} \yrec{2}{3} \yrec{2}{3} \yrec{4}{4}\yrec{2}{4}
\yrec{4}{1} \yrec{4}{6} \yrec{2}{8} \yrec{4}{8}
\filldraw (0,-2.2) rectangle (.2,-2.8);
\filldraw (7.2,-2.2) rectangle (8,-2.8);
\draw[ultra thick,ultragreen, dashed] (3,-4.5)rectangle(4,.5);
\draw[ultra thick,ultragreen, dashed] (7,-4.5)rectangle(8,.5);
\end{tikzpicture} \raisebox{1.3cm}{$\quad \mathlarger{\mathlarger{,}}  \quad$}
\begin{tikzpicture}[scale=.45]
\draw(0,0)grid(8,-4);
\rrec{1}{1} \rrec{2}{6} \rrec{3}{4} \rrrec{4}{3}
\yrec{1}{2} \yrec{2}{2} \yrec{3}{5} \yrec{4}{5} \yrec{4}{2}
\yrec{2}{5} \yrec{3}{5} \brec{1}{4}{5} \brec{3}{1}{3} \yrec{2}{3} \yrec{2}{4}
\yrec{4}{1} \yrec{4}{6} \brec{3}{7}{2} \yrec{2}{7} \yrec{4}{7} \yrec{2}{8} \yrec{4}{8}
\filldraw (0,-2.2) rectangle (.2,-2.8);
\filldraw (7.8,-2.2) rectangle (8,-2.8);
\draw[ultra thick,ultragreen, dashed] (6,-4.5)rectangle(7,.5);
\draw[ultra thick,ultragreen, dashed] (7,-4.5)rectangle(8,.5);
\end{tikzpicture} 
   
    \caption{We can add black-yellow-black-yellow columns to the $6$-orbit for $\overline{F}(4,1,4,1)$ get orbits for $\overline{F}(a,1,a,1)$, $a\geq4$.}
    \label{fig:4dilation}
\end{figure}

\begin{thm} For $a\geq 2$, row motion on $\overline{F}(a,1,a,1)$ has $a-2$ small orbits $\OO_s$ of size $a+2$ with $\MM(\OO_s)=2a+2$ and $2$ large orbits $\OO_l$ of size $2a+3$ with $\MM(\OO_l)=4a+2$. The statistic $\overline{\chi}$ is $(a+1)$-mesic.
\end{thm}
\begin{proof} The orbits for the case $a=3$ were already examined in Example \ref{ex:3131}. For larger $a$, all small orbits have one red tile starting at each row. For any $0\leq k\leq a-4$ we get a unique $5$-orbit with red tiles starting on rows $1,2,3,4$ placed on columns $1$, $6+k$, $4+k$ and $3$ respectively. All these orbits can be obtained by dilating the $6$-orbit of rowmotion on $\overline{F}(4,1,4,1)$ shown on Figure \ref{fig:4dilation}, left by adding columns with alternating black and yellow tiles. The only other $a-2$ orbit has red tiles starting on rows $1$ and $3$ on column $1$, and red tiles starting on rows $2$ and $4$ on column $3$.

The larger orbits can similarly be obtained by dilating the $9$-orbits shown in Example \ref{ex:3131}. The placement of starting position of red tiles on one of the large orbits is : columns $1$ and $a+3$ on row $1$, $4$ in row $2$,$a+4$ in row $3$ and $3$ in row $4$. The other orbit is obtained by shifting rows cyclically by $2$, so the positions for rows $1$ and $2$ are flipped with the positions for rows $3$ and $4$ respectively.

As all $(a+2)^2$ elements are represented, there can be no more orbits. Equations \ref{eq:maxorbit}-\ref{eq:sumorbit} give the data about the statistics.
\end{proof}

\begin{wrapfigure}{r}{0.35\linewidth}\centering
\begin{tikzpicture}[scale=.45]
\draw (-4,0) rectangle(4,-4);
\draw (-4,0) grid (-3,-4);
\draw (-1,0)grid(4,-4);
\rrec{1}{2}\rrrec{4}{4}\yrec{1}{3}\yrec{2}{1}\yrec{2}{3}\yrec{2}{4}\yrec{4}{1}\yrec{2}{-3}\yrec{4}{-3}\yrec{2}{0}\yrec{4}{0}
\yrec{4}{2}\yrec{4}{3}
\brec{3}{-.5}{5.5}\brec{1}{-.5}{2.5}
\brec{3}{-3}{3.5}\brec{1}{-3}{3.5}
\draw[ultra thick,dashed](-3,-1.5)--(-1,-1.5)(-3,-3.5)--(-1,-3.5);\draw[thick,yellow,dashed](-3,-1.5)--(-1,-1.5)(-3,-3.5)--(-1,-3.5);
\draw [
    thick,
    decoration={
        brace,
        mirror,
    },
    decorate
] (-4,-4.2) -- (1,-4.2);
\node at (-1.5,-4.8) {$k-1$};
\end{tikzpicture} \end{wrapfigure} 
For a fixed $k$ value, we can obtain all orbits of $\overline{F}(k,1,a+k+2,1)$ from orbits of $\overline{F}(k,1,a,1)$ by adding new pieces that extend the black tiles on the third row, while adding an extra black tile to row two. The addition of the $4\times (k+2)$ piece on the right to extend each black tile in row $3$, shifting cyclically if necessary achieves exactly this. We will now use this process to calculate the orbits of rowmotion on $\overline{F}(1,1,a,1)$.

\begin{thm} If $a\equiv 0$ or $2$ mod $3$, rowmotion on $\overline{F}(a,1,1,1)$ has a unique orbit $\OO$ of size $3a+4$ and $\overline{\chi}$ is homomesic. If $a\equiv 1$ mod $3$, then rowmotion has $3$ orbits, of sizes $a+2$, $a+1$ and $a+1$ and $\overline{\chi}$ values $(a+2)(a+3)/2$, $((a+2)(a+3)/2$ and $((a)(a+3)/2$ respectively.
\end{thm}
\begin{proof} For the cases when $a\equiv 0$ or $2$ mod $3$, it is possible to extend the unique orbits for $\overline{F}(1,1,2,1)$  and  $\overline{F}(1,1,3,1)$ (see Figure ~\ref{fig:rowmotionona111}) to get the orbits for $\overline{F}(1,1,2+3t,1)$ and $\overline{F}(1,1,3+3t,1)$, $t\in\NN$. Considering the sizes shows us that no other orbits exist.  Similarly, the three orbits of $\overline{F}(1,1,1,1)$ can be extended to get orbits for $\overline{F}(1,1,1+3t,1)$ as seen in Figure~\ref{fig:rowmotion2ona111}. From Table~\ref{tab:smallcases} we can see that $\overline{F}(1,1,a,1)$ has a total of $3a+4$ ideals, so these are all the orbits. To calculate the $\overline{\chi}$ values, we can use Equation~\ref{eq:sumorbit}, which simplifies to $\overline{\chi}(\OO)=(a+3)|\OO|/2 +(r_1-r_4)+(a+1)/2(r_3-r_2)$ for the particular case of $\overline{F}(1,1,a,1)$.
\end{proof}

\begin{figure}[ht]
    \centering
\begin{tikzpicture}[scale=.45]
\draw (0,0)grid(10,-4);
\rrec{1}{1}\rrec{3}{1}\rrec{2}{3}\rrrec{4}{3}\rrec{1}{4}\rrec{1}{6}\rrec{3}{5}\rrrec{4}{8}\rrrec{4}{10}\rrec{2}{9}
\yrec{1}{2}\yrec{2}{2}\yrec{3}{2}\yrec{4}{2}\yrec{4}{4}\yrec{1}{5}\yrec{2}{5}\yrec{3}{6}\yrec{4}{6}\yrec{1}{7}\yrec{2}{7}\yrec{4}{7}\yrec{2}{8}\yrec{3}{8}\yrec{1}{9}\yrec{4}{9}\yrec{2}{10}
\brec{3}{4}{1}\brec{3}{7}{1}\brec{3}{10}{1}
\draw[ultra thick,ultragreen, dashed] (3,-4.5)--(3,.5) ;
\draw[ultra thick,ultragreen, dashed] (7,-4.5)--(7,.5) ;
\draw[ultra thick,ultragreen, dashed] (10,-4.5)--(10,.5) ;
\end{tikzpicture}\quad
\begin{tikzpicture}[scale=.45]
\draw (0,0)grid(13,-4);
\rrec{1}{1}\rrec{3}{1}\rrec{2}{3}\rrrec{4}{3}\rrec{1}{4}\rrec{2}{7}\rrec{1}{9}\rrrec{4}{8}\rrrec{4}{6}\rrec{3}{10}\rrec{1}{11}\rrrec{4}{13}
\yrec{1}{2}\yrec{2}{2}\yrec{3}{2}\yrec{4}{2}\yrec{4}{4}\yrec{1}{5}\yrec{2}{5}\yrec{4}{5}\yrec{3}{6}\yrec{2}{6}\yrec{1}{7}\yrec{2}{8}\yrec{4}{7}\yrec{4}{9}\yrec{1}{10}\yrec{2}{10}\yrec{4}{11}\yrec{3}{11}\yrec{2}{13}\yrec{1}{12}\yrec{2}{12}\yrec{4}{12}
\brec{3}{4}{2}\brec{3}{8}{2}\brec{3}{12}{2}
\draw[ultra thick,ultragreen, dashed] (3,-4.5)--(3,.5) ;
\draw[ultra thick,ultragreen, dashed] (8,-4.5)--(8,.5) ;
\draw[ultra thick,ultragreen, dashed] (12,-4.5)--(12,.5) ;
\end{tikzpicture}\raisebox{1.1cm}{$ \quad \mathlarger{\mathlarger{\mathlarger{+}}} \quad  $}
\begin{tikzpicture}[scale=.45]
\draw (0,0)grid(3,-4);
\rrec{1}{1}\rrrec{4}{3}
\yrec{1}{2}\yrec{2}{2}\yrec{4}{1}\yrec{4}{2}\yrec{2}{3}\brec{3}{1}{3}
\draw[ultra thick,white, dashed] (3,-4.5)--(2,-4.5) (3,.5)--(2,.5);
\end{tikzpicture}

    \caption{Orbits for $\overline{F}(1,1,2,1)$ (left) and  $\overline{F}(1,1,3,1)$ (middle) can be extended via the piece on the right to get orbits for $\overline{F}(1,1,2+3t,1)$ and $\overline{F}(1,1,3+3t,1)$ }
    \label{fig:rowmotionona111}
\end{figure}
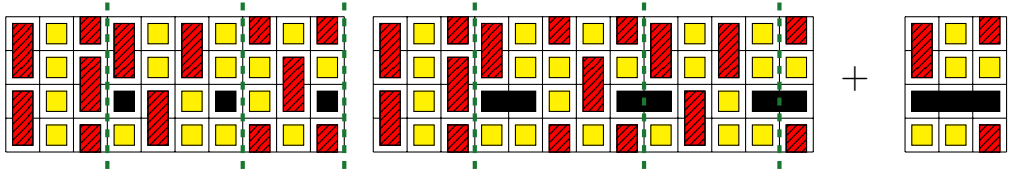

\begin{figure}
    \centering
\begin{tikzpicture}[scale=.45]
\draw (0,0)grid(3,-4);
\rrec{1}{1}\rrec{3}{1}\rrec{2}{3}\rrrec{4}{3}
\yrec{1}{2}\yrec{2}{2}\yrec{3}{2}\yrec{4}{2}
\draw[ultra thick,ultragreen, dashed] (3,-4.5)--(3,.5) ;
\end{tikzpicture}\quad
\begin{tikzpicture}[scale=.45]
\draw (0,0)grid(2,-4);
\rrec{1}{1}\rrec{3}{2}
\yrec{1}{2}\yrec{2}{2}\yrec{3}{1}\yrec{4}{1}
\draw[ultra thick,ultragreen, dashed] (1,-4.5)--(1,.5) ;
\end{tikzpicture}\quad
\begin{tikzpicture}[scale=.45]
\draw (0,0)grid(2,-4);
\rrrec{4}{1}\rrec{2}{2}
\yrec{1}{2}\yrec{4}{2}\yrec{3}{1}\yrec{2}{1}
\draw[ultra thick,ultragreen, dashed] (1,-4.5)--(1,.5) ;
\end{tikzpicture}\raisebox{1.1cm}{$ \quad \mathlarger{\mathlarger{\mathlarger{+}}} \quad  $}
\begin{tikzpicture}[scale=.45]
\draw (0,0)grid(3,-4);
\rrec{1}{1}\rrrec{4}{3}
\yrec{1}{2}\yrec{2}{2}\yrec{4}{1}\yrec{4}{2}\yrec{2}{3}\brec{3}{1}{3}
\draw[ultra thick,white, dashed] (3,-4.5)--(2,-4.5) (3,.5)--(2,.5);
\end{tikzpicture}\raisebox{1.1cm}{$ \quad \mathlarger{\mathlarger{\mathlarger{\rightarrow}}} \quad  $}\begin{tikzpicture}[scale=.45]
\draw (0,0)grid(6,-4);
\rrec{1}{1}\rrec{3}{1}\rrec{2}{3}\rrrec{4}{3}
\yrec{1}{2}\yrec{2}{2}\yrec{3}{2}\yrec{4}{2}
\rrec{1}{4}\rrrec{4}{6}
\yrec{1}{5}\yrec{2}{5}\yrec{4}{4}\yrec{4}{5}\yrec{2}{6}\brec{3}{4}{3}
\draw[ultra thick,white, dashed] (3,-4.5)--(2,-4.5) (3,.5)--(2,.5);
\end{tikzpicture}\quad\begin{tikzpicture}[scale=.45]
\draw (0,0)grid(5,-4);
\rrec{1}{1}\rrec{3}{5}
\yrec{1}{5}\yrec{2}{5}\yrec{3}{1}\yrec{4}{1}
\rrec{1}{4}\rrrec{4}{3}
\yrec{1}{2}\yrec{2}{2}\yrec{4}{4}\yrec{4}{2}\yrec{2}{3}\brec{3}{2}{3}
\draw[ultra thick,white, dashed] (3,-4.5)--(2,-4.5) (3,.5)--(2,.5);
\end{tikzpicture}\raisebox{1.1cm}{$ \quad \mathlarger{\mathlarger{\mathlarger{\cdots}}}$}

    \caption{Orbits for $\overline{F}(1,1,1,1)$ (left)  can be extended to get orbits for ,$\overline{F}(1,1,1+3t,1)$ }
    \label{fig:rowmotion2ona111}
\end{figure}
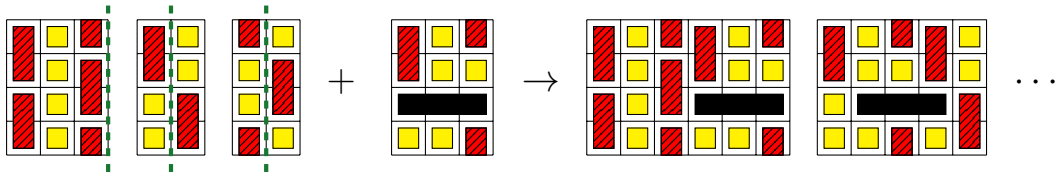
A similar process shows that for $a\geq 2$, the number of orbits of rowmotion on $\overline{F}(2,1,a,1)$ depends on what $a$ is in modulo $4$. If $a\equiv 1$ we get a unique orbit of size $4a+6$. If $a \equiv 3$, we get three orbits, two of size $a+1$ and one of size $2a+4$. If $a$ is even, we get two orbits of sizes $a+2$ and $3a+4$ respectively. The statistics for this orbits can be found in Table~\ref{tab:rowmotion}. As mentioned before, a similar process can be used to characterize orbits of fences of size $(k,1,a,1)$ where even more cases would be involved.

\begin{table}\center
\begin{tabular}{||c|c|c|c|c||}
\hline
Composition& Orbit Count& $|\OO|$& $\MM(\OO)$&$\overline{\chi}(\OO)$ \\ [0.5ex] 
 \hline\hline  
\multirow{2}{*}{\begin{tabular}{c}
$(a,b)$:  \\ gcd$(a,b)=m$
\end{tabular}}&  $1$&$m+2$&$2m(a+b)(m+2)/ab$&$|\OO|n/2$\\\cline{2-5}
     & $(ab/m) -1$ &$m$&$2m-(a+b)m/ab$&$|\OO|n/2$\\\hline
%\multirow{1}{*}{$(a=3t-1,1,1,1)$}&  $1$&$3a+4$&$5a+6$&$|\OO|n/2$\\\hline
\multirow{1}{*}{$(a\neq3t+1,1,1,1)$}&  $1$&$3a+4$&$5a+6$&$|\OO|n/2$\\\hline
\multirow{3}{*}{$(a=3t+1,1,1,1)$}&  $1$&$a+2$&$5t+4$&$|\OO|n/2$\\\cline{2-5}
&  $1$&$a+1$&$5t+2$&$(|\OO|+1)n/2$\\\cline{2-5}
&  $1$&$a+1$&$5t+2$&$(|\OO|-1)n/2$\\\hline
\multirow{2}{*}{$(a=4t-2,1,2,1)$}&  $1$&$a+2$&$7t-1$&$|\OO|n/2$\\\cline{2-5}&  $1$&$3a+4$&$21t-7$&$|\OO|n/2$\\\hline  
\multirow{3}{*}{$(a=4t-1,1,2,1)$}&   $1$&$a+1$&$7t-1$&$|\OO|n/2-(a+1)/2$\\\cline{2-5}& 
$1$&$a+1$&$7t-1$&$|\OO|n/2+(a+1)/2$\\\cline{2-5}&
$1$&$2a+4$&$14t+1$&$|\OO|n/2$\\\hline  
\multirow{2}{*}{$(a=4t,1,2,1)$}&  $1$&$a+2$&$7t+2$&$|\OO|n/2$\\\cline{2-5}&  
$1$&$3a+4$&$21t+4$&$|\OO|n/2$\\\hline 
\multirow{1}{*}{$(a=4t+1,1,2,1)$}&  $1$&$4a+6$&$7a+6$&$|\OO|n/2$\\\hline  
\multirow{2}{*}{$(a,1,a,1)$}&  $a-2$&$a+2$&$2a+2$&$|\OO|n/2$\\\cline{2-5}
&  $2$&$2a+3$&$4a+2$&$|\OO|n/2$\\\hline  
\multirow{5}{*}{$(a,a,a,a)$}&  $a^3-4a^2+6a-3$&$a$&$4a-4$&$|\OO|n/2$\\\cline{2-5}
&  $a$&$a+1$&$4a-2$&$(|\OO|-1)n/2$\\\cline{2-5} 
&  $a$&$a+1$&$4a-2$&$(|\OO|+1)n/2$\\\cline{2-5} 
&  $1$&$a+2$&$4a$&$|\OO|n/2$\\\cline{2-5}
&  $2a-2$&$2a^2$&varied&varied \\\hline  
\end{tabular}
\caption{The behaviour of rowmotion on circular fences for small examples, where $n$ denotes the size of $\alpha$}
\label{tab:rowmotion}
\end{table}

Looking at Table~\ref{tab:rowmotion}, we see that even when we focus on the cases with at most $4$ parts, it is difficult to predict the numbers and lengths of the orbits in general. The connection to modular arithmetics is prevalent, but it is not always as simple as looking at the gcd of the parts, as in the case of two parts. When we consider partitions of type $(a,1,k,1)$ with $a>k$ for example, the structure seems to depend on $\text{gcd}(a,k+2)$ instead. Note that the statistic $\overline{\chi}$ is quite well behaved, with orbits or pairs of orbits averaging out to $|\OO|n/2$ where $n=|\alpha|$.

\subsection{The case of $(a,a,a,a)$ and a note on orbomesy}

 When a statistic has the same average on all orbits of the same size, it is called \emph{orbomesic}. Extending the idea of homomesy, the orbomesy phonemenon is introduced in \cite{rowmotion} and illustrated through a number of cases it applies to. The rowmotion on fences is a periodic operation, and except when we get shared elements it applies to each segment independently according to their own size. The orbit structure therefore is determined by how often we get shared elements- how in sync the action on different segments are.  In the examples of orbomesy given in \cite{rowmotion}, we often get groups of orbits that, though not isomorphic in a well defined sense, are structurally equivalent and are formed by picking different pairings of moduli that are out of sync. As a result, they naturally have the same length, $\MM$ value and $\overline{\chi}$ value, resulting in an apparent orbomesy. In the circular case, we see that the orbomesy of $\overline{\chi}$ breaks down completely, in that we either get a full homomesy or we get pairs of orbits of the same size with different $\overline{\chi}$ values. Now we will look at the case of $\overline{F}(a,a,a,a)$ where this is especially visible. 
\begin{thm} Rowmotion on $\overline{F}(a,a,a,a)$ has:
\begin{itemize}
\item $a^3-4 a^2 +6a -3$ orbits of size $a$, satisfying $\MM(\OO)=4a-4$, $\overline{\chi}(\OO)=2a|\OO|$,
\item $2a$ orbits of size $a+1$ with $\MM(\OO)=4a-2$, $a$ of them satisfying $\overline{\chi}(\OO)=2a(|\OO|+1)$, the other $a$ satisfying  $\overline{\chi}(\OO)=2a(|\OO|-1)$.
\item $1$ orbit of size $a+2$,  $\MM(\OO)=4a$, $\overline{\chi}(\OO)=2a|\OO|$,
\item $2a-2$ orbits of size $2a^2$ with $\MM(OO)=8a^2-10a+4$, where for each $r \in \{0,1,\ldots,a-2\}$ we get two orbits whose $\overline{\chi}$ value is equal to $4a^3+2a^2-4a+4ra$.
\end{itemize}

\end{thm}
\begin{proof} We will describe each of these orbits. As the total number of ideals represented, $a^4+4a^2+2$, matches $\overline{R}((a,a,a,a);1)$ there can be no other orbits. After describing the orbits, the statistics can be calculated via the simplified formulas:$$\MM(\OO)=\sum_{i\leq4}b_i(a-1)+r_i\qquad \overline{\chi})(\OO)=a(2|\OO|+r_1-r_2+r_3-r_4).$$
\begin{itemize}
\item The size $a$ orbits are the ones that contain no red tile. Each row contains one white and one black tile, where white tiles on consecutive rows fall on different columns, including rows $1$ and $4$. Placing the white tile on position $1$ of the first row, the rest of the white tiles can be placed in $(a-1)^3-(a-1)(a-2)=a^3-4 a^2 +6a -3$ ways.
\item The size $a+1$ orbits have two red tiles that lie in different columns, which we can choose in $a$ ways. They either start on rows $1$ and $3$ or rows $2$ and $4$, giving us $2a$ such orbits in total.  
\item The unique size $a+2$ orbit contains $4$ red tiles: starting at rows $1$ and $3$ of column $1$ and rows $2$ and $4$ of column $3$.
\item The largest orbits can be indexed with $r \in \{0,1,\ldots,a-2\}$, where we get a pair of orbits $\OO^1_r$ and $\OO^2_r$  of size $2a^2$ for each choice of $r$. $\OO^1_r$ has the following positions for the red tiles:

\hfill\begin{minipage}{\dimexpr\textwidth-3cm}
\begin{itemize}
    \item [ROW 1:] $1+(a+1)t$ for  $0\leq t\leq r$,
    \item [ROW 2:] $a^2-(a+1)t$ for  $1\leq a-1-r$,
    \item [ROW 3:] $a^2+(a+1)t$ for   $0\leq t\leq r$,
    \item [ROW 4:] $2a^2-(a+1)t$ for  $1\leq a-1-r$.
\end{itemize}
\xdef\tpd{\the\prevdepth}
\end{minipage} $\OO^2_r$ has the values for rows $2$ and $4$ flipped. \end{itemize}\end{proof}

The value $\MM$ seems to be orbomesic in this example, in fact that is the case in all examples listed on Table~\ref{tab:rowmotion}. That could be indicative of a general pattern or could be because we are only looking at a very limited sample of examples. For $(a,a,a,a)$, the $\overline{\chi}$ values of orbits of size $2a^2$ are not just paired up but actually follow an arithmetic progression centered around $|\OO|n/2$. It would be interesting to see if this trend continues in larger examples.
\section{Comments, Questions and Future Directions}\label{sec:further}
We list some questions and observations here that are of natural interest.
\begin{itemize} 
\item \textbf{Bijective proofs:} The original starting point of this work was finding a bijective proof for the symmetry of the rank sequences of lower ideas of circular fences. The setwise complement of a size $k$ lower ideal is a size $n-k$ upper ideal. : The tantalizingly simple notion of taking this setwise complement and then letting the beads (nodes belonging to the ideal) fall with gravity unfortunately did not work. When there are filled or empty sections, the algorithm can not be described locally section by section, which makes both for a tricky description and a complicated proof. It is possible however, that another perspective on the objects might lead to a more natural approach to the proof. 

Fence posets are in bijective correspondence in a variety of combinatorial objects. An example is given by perfect matchings in snake graphs. Circular fence posets, similarly, can be viewed as a circular analogue to snake graphs, with two ends identified. It is possible to get a bijection if two ends are identified in a \emph{parity reversing} way to avoid any extra matchings forming, or disallowing matchings that do not work in the uncircular case as done in \cite{bandgraphs}.   The natural symmetries of this object are different and can possibly provide new insight.

\item \textbf{Rowmotion orbits under shifting:} We have seen in Lemma~\ref{lem:rowmotionflip} that the setwise complement map $\kappa$ gives a natural bijection between orbits of $\overline{F}(\alpha)$ and $\overline{F}(\text{sh}(\alpha))$ that takes $\overline{\chi}$ to $n|\OO|-\overline{\chi}$, while fixing orbit length and $\MM$. The pairing up of $\overline{\chi}$ statistics seen in Table~\ref{tab:rowmotion} suggests that it might be possible to find a bijection that also \emph{fixes} $\overline{\chi}$. That would be exciting in two levels. It would confirm that in the circular case, $\overline{\chi}$ is never truly orbomesic, it is either a true homomesy or we have orbits of the same size with different $\overline{\chi}$ values. It would also provide a bijective proof for the symmetry of the rank polynomial, possibly making way to a bijective proof of unimodality in the non-circular case. 

\item \textbf{A Polyhedral Perspective: } A related and probably simpler question that we were unable to answer goes as follows. 
Given a composition $\alpha$ of $n$, consider the polytope $\overline{P}_{\alpha} \subset \mathbb{R}^n$ given by the indicator vectors of $\overline{J}(\alpha)$, the set of all lower ideals of the associated circular fence poset. Consider the sections of the polytope:
\[\overline{P}_{\alpha}^t = \overline{P}_{\alpha} \cap \{x\in \mathbb{R}^n, \, \sum_{i} x_i = t\}.\]
We have observed that the function $t \rightarrow \operatorname{Vol} \, \overline{P}_{\alpha}^t$
is symmetric about the point $n/2$, that is, \[\operatorname{Vol}(\overline{P}_{\alpha}^t) = \operatorname{Vol}(\overline{P}_{\alpha}^{n/2-t}),\qquad 0 \leq t \leq n.\]
Interestingly, these polytopes are not necessarily  combinatorially equivalent. The special case when we look at compositions $(\alpha_1, 1, \ldots, \alpha_s, 1)$ of $n$ is especially interesting. Note that thesee are of special interest (at least to the authors) as these are potentially the only compositions where we are yet to settle the unimodality problem for rank sequences. In this case, we can parse the question as follows. Define the polytope $H_{\alpha}$ by 
\[H_{\alpha} = \{x \in \prod_{i = 1}^{s} [0, \alpha_i + 1], \,\,\, x_i - x_{i+1\, (\operatorname{mod}\,s)\,} \leq \alpha_i, \, i = 1, \ldots, s\}.\]
Then the above conjecture in this special reduces to the claim that 
\[\operatorname{Vol}(\overline{H}_{\alpha}^t) = \operatorname{Vol}(\overline{H}_{\alpha}^{n/2-t}), \qquad 0 \leq t \leq n,\] where these polytopes are defined similarly to above. Again, we have equality of volumes despite the polytopes not necessarily being isomorphic. A natural explanation of this would be interesting. \item \textbf{Refinements of Unimodality: } In their paper \cite{Saganpaper}, McConville, Sagan and Smyth investigated the existence of chain decompositions as a possible method of proving unimodality. Though unable to make process in this direction, the examples we considered led us the believe that for circular fences -apart from the case $\alpha = (a, 1, a, 1)$ (see Figure~\ref{fig:1417})- the associated lattices admit \emph{symmetric chain decompositions} and are thus \emph{strongly Sperner}. A resolution of this would be satisfying and would go some way towards elucidating the structure of fence and circular fence posets. 

\item\textbf{Skew Young-Posets:} The boxes on the Ferrers diagram of a partition $\lambda$ have a natural poset structure, where ideals are in bijection with partitions whose Ferrers diagrams fit inside $\lambda$. From this viewpoint, fence poset can be viewed as the posets for certain skew-diagrams $\lambda\backslash \mu$ corresponding to maximal border strips. The unimodality of the corresponding rank polynomial in the non-skew case was studied previously by Stanton in 1990 (see \cite{stanton}), where he conjectured that self dual partitions give rise to unimodal polynomials. He also provided several examples where unimodality fails. These examples are of a similar flavor to the examples in the circular case in that the two largest entries are seperated by a slightly  smaller entry which provides the only violation to unimodality.

Progress towards a general classification has been limited since then. Zbarsky showed in \cite{nearrectangular} that when the partition $\lambda$ is satisfies certain properties to ensure the parts are of similar sizes the rank polynomial is unimodal. He also provided further examples where unimodality fails and conjectured that in any example that unimodality fails the rank polynomial is bimodal with the two modes being seperated by one entry only.

The tricky part about the case of the (possibly skew) diagrams compared to the fences is that the position of the mode (or modes) is trickier to determine and does not necessarily lie around $q^{|\lambda/2|}$. Nevertheless, there is enough similarity to warrant a new look at this problem through the lens of skew-diagrams and possible circular analogues.

\item \textbf{Extremal Ranks: } For a fixed size, how does changing the shape of the fence affect the resulting rank sequence? It is a simple observation to see that any maximum is achieved at a composition with parts $\leq 2$:
\begin{prop} Let $\alpha'$ be obtained from a partition $\alpha$ from replacing a part of size $t\geq3$ by parts $t-2,1,1$. Then $\overline{R}(\alpha';q)-\overline{R}(\alpha;q)$ has non-negative integer coefficients.
\end{prop}
\begin{proof} As the rank polynomial is invariant under cyclic shifts, we can assume that the part of size $t$ is the first part. Let $x\preceq y \preceq z$ be the maximum nodes on the first segment. The ideals of $\alpha'$ can be obtained by replacing these relations with the weaker set $ y\succeq x \preceq z$, where $x$ and $z$ are incomparable.
\end{proof}

Our experiments suggest that following is true. 
\begin{itemize}
    \item Given any composition $\alpha$ of $n$, we conjecture that $r(\alpha) \leq r(1^n)$, where the inequalities are pointwise i.e. we believe that for every $k$, the number of rank $k$ down ideals of $F(\alpha)$ are at most the number of rank $k$ down ideals of $F(1^n)$. \item More generally, we conjecture that for any fixed $k$ and $n$ where $k$ divides $n$ and any composition of $n$ with $k$ parts, we have that $r(\alpha) \leq r(n/k, \ldots, n/k)$. 
\end{itemize}

\end{itemize}

%\item\textbf{Higher Dimensional Fence Posets:}There is a natural way to visualize fence posets using borders of young diagrams. Given a young diagram $\lambda = (\lambda_1 \geq \ldots, \lambda_k > 0)$, we think of it as a collection of square boxes whose top right vertices are given by
%\[\lambda \sim \{(a, b) \in \mathbb{N}^2,  \, 1 \leq a \leq \lambda_b, \, 1 \leq b \leq k\}.\]
%This is a region in the plane : Let $S(\lambda)$ be the set of corners with both co-ordinates in $\mathbb{N}$, viewed as lattice points, visible from far out in the positive orthant. We define an order relation on elements in $S(\lambda)$ by $(a, b) < (c, d)$ provided $a \leq c$ and $b \leq d$. 
%For instance, the young diagram $(4, 4, 4, 4, 3)$ drawn in Russian style below corresponds to the poset marked with a solid line on the diagram. The order relations correspond to the natural order on $\mathbb{N}^2$ and it is clear that this is a realization of the fence poset $(2, 1, 1, 3)$. 

%\begin{center}
    %\begin{tikzpicture}[scale=.5]
     %   \draw[gray] (0,0) --++ (5,-5);
 %       \draw[gray] (1,1) --++ (5,-5);
   %     \draw[gray] (2,2) --++ (5,-5);
    %    \draw[gray] (3,3) --++ (5,-5) (5,3) --++ (4,-4)  \draw[gray] (1,-1) --++ (4,4) (2,-2) --++ (4,4) (3,-3) --++ (4,4) (4,-4) --++ (4,4) (5,-5) --++ (4,4); 
      %\draw[very thick] (1,1) --++ (2,2) --++ (1, -1) --++ (1,1) --++ (3, -3) ;
 %   \end{tikzpicture}
%\end{center}

\section*{Acknowledgements}
The authors would like to thank Bruce Sagan for several helpful comments on a preliminary version of this paper. Both authors were supported by the grant SUP-17483 from the Bogazici University Scientific Research Office.

%\bibliographystyle{alpha}
%\bibliography{main}

\begin{thebibliography}{99}
%\begin{thebibliography}{I{\c{C}anak\c{c}{\i}}S19}

\bibitem[Cla20]{claussen2020expansion}
Andrew Claussen.
\newblock Expansion posets for polygon cluster algebras, 2020.

\bibitem[EPRS21]{rowmotion}
Sergi Elizalde, Matthew Plante, Tom Roby, and Bruce Sagan.
\newblock Rowmotion on fences, 2021.

\bibitem[CS19]{bandgraphs}
\.{I}lke {\c{C}anak\c{c}{\i}} and Ralf {Schiffler}.
\newblock {Snake graph calculus and cluster algebras from surfaces. III: Band
  graphs and snake rings}.
\newblock {\em {Int. Math. Res. Not.}}, 2019(4):1145--1226, 2019.

\bibitem[MGO20]{originalconj}
Sophie Morier-Genoud and Valentin Ovsienko.
\newblock {$q$}-deformed rationals and {$q$}-continued fractions.
\newblock {\em Forum Math. Sigma}, 8:Paper No. e13, 55, 2020.

\bibitem[MS02]{crown}
Emanuele Munarini and Norma Salvi.
\newblock On the rank polynomial of the lattice of order ideals of fences and
  crowns.
\newblock {\em Discrete Mathematics}, 259:163--177, 12 2002.

\bibitem[MSS21]{Saganpaper}
Thomas {McConville}, Bruce~E. {Sagan}, and Clifford {Smyth}.
\newblock {On a rank-unimodality conjecture of Morier-Genoud and Ovsienko},
  2021.
\newblock Id/No 112483.

\bibitem[Mun06]{crown2}
Emanuele Munarini.
\newblock A combinatorial interpretation of the chebyshev polynomials.
\newblock {\em SIAM Journal on Discrete Mathematics}, 20(3):649--655, 2006.

\bibitem[Sta90]{stanton}
Dennis {Stanton}.
\newblock {Unimodality and Young's lattice}.
\newblock {\em {J. Comb. Theory, Ser. A}}, 54(1):41--53, 1990.

\bibitem[Sta12]{Stanley}
Richard~P. {Stanley}.
\newblock {\em {Enumerative combinatorics. Vol. 1.}}, volume~49.
\newblock Cambridge: Cambridge University Press, 2012.

\bibitem[Zba15]{nearrectangular}
Samuel {Zbarsky}.
\newblock {Unimodality of partitions in near-rectangular Ferrers diagrams}.
\newblock {\em {Discrete Math.}}, 338(9):1649--1658, 2015.

\end{thebibliography}
\end{document}